\documentclass[10pt,reqno]{amsart}

\usepackage{subfigure}
\usepackage{graphics}
\usepackage{psfrag}
\usepackage{epsfig}
\usepackage{pstricks}
\newtheorem{theorem}{Theorem}[section]
\newtheorem{lemma}[theorem]{Lemma}

\theoremstyle{definition}

\theoremstyle{remark}
\newtheorem{remark}[theorem]{Remark}

\numberwithin{equation}{section}

\begin{document}

\title[Correctors and Field Fluctuations...]{Correctors and Field Fluctuations for the $p_{\epsilon}(x)$-Laplacian with rough exponents}

\author[Silvia Jimenez \and Robert P. Lipton]{Silvia Jimenez \and Robert P. Lipton\\\\\tiny{Dept. of Mathematics, Louisiana State University, Baton Rouge, LA 70803, USA\\phone: +1.225.578.1665, fax: +1.225.578.4276}\\\tiny{E-mail: sjimenez@math.lsu.edu (S. Jimenez), lipton@math.lsu.edu (R. Lipton)}}

\thanks{This research was supported by the NSF Grant DMS-0807265 and the AFOSR Grant FA9550-08-1-0095}

\keywords{p-laplacian, power-law, homogenization, correctors, layered media, dispersed media, periodic domain}

\subjclass[2000]{Primary 35J66; Secondary 35A15, 35B40, 74Q05}

\textit{}
\begin{abstract}

We provide a corrector theory for the strong approximation of fields inside composites made from two materials 
with different power law behavior.  The correctors are used to develop bounds on the local singularity strength 
for gradient fields inside micro-structured media. The bounds are multi-scale in nature and can be used to measure 
the amplification of applied macroscopic fields by the microstructure. 
\end{abstract}


\maketitle

\section{Introduction}

In this article we consider boundary value problems associated with fields inside heterogeneous materials made from 
two power-law materials.  The geometry of the composite is periodic and is specified by the indicator function of the 
sets occupied by  each of the materials.  The indicator function of material one and two are denoted by $\chi_1$ and 
$\chi_2$, where $\chi_{1}(y) =1$ in material one and is zero outside and $\chi_{2}(y)=1-\chi_{1}(y)$.  The constitutive 
law for the heterogeneous medium is described by $A:\mathbb{R}^{n}\times\mathbb{R}^{n}\rightarrow\mathbb{R}^{n}$,
\begin{equation}
	\label{A}
	\displaystyle
A\left(y,\xi\right)=\sigma(y)\left|\xi\right|^{p(y)-2}\xi,
\end{equation}	
with $\sigma(y)=\chi_{1}\left(y\right)\sigma_{1}+\chi_{2}\left(y\right)\sigma_{2}$, and 
$p(y)=\chi_{1}\left(y\right)p_{1}+\chi_{2}\left(y\right)p_{2}$, periodic in $y$, with unit period cell $Y=(0,1)^n$.  
This simple constitutive model is used in the mathematical description of many physical phenomena including plasticity \cite{PonteCastaneda1997,PonteCastaneda1999,Suquet1993,Idiart2008}, nonlinear dielectrics \cite{Garroni2001,Garroni2003,Kohn1998,Talbot1994,Talbot1994-2}, and fluid flow \cite{Ruzicka2000,Antontsev2006}.  We 
study the problem of periodic homogenization associated with the solutions $u_{\epsilon}$ to the problems
\begin{equation}
	\label{01}
	\displaystyle
-div\left(A\left(\frac{x}{\epsilon},\nabla u_{\epsilon}\right)\right)=f \text{ on $\Omega$, $u_{\epsilon}\in W_{0}^{1,p_{1}}(\Omega),$}
\end{equation}
where $\Omega$ is a bounded open subset of $\mathbb{R}^{n}$, $2\leq p_{1}\leq p_{2}$, $f\in W^{-1,q_{2}}(\Omega)$, and 
$1/p_{1}+1/q_{2}=1$.  The differential operator appearing on the left hand side of (\ref{01}) is commonly referred to 
as the $p_{\epsilon}(x)$-Laplacian. For the case at hand, the exponents $p(x)$ and coefficients $\sigma(x)$ are  taken 
to be simple functions.  Because the level sets associated with these functions can be quite general and irregular they 
are referred to as rough exponents and coefficients.  In this context all solutions are understood in the usual weak 
sense \cite{Zhikov1994}.

One of the basic problems in homogenization theory is to understand the asymptotic behavior as $\epsilon\rightarrow0$, 
of the solutions $u_{\epsilon}$ to the problems (\ref{01}).  It was proved in \cite{Zhikov1994} that 
$\{u_{\epsilon}\}_{\epsilon>0}$ converges weakly in $W^{1,p_{1}}(\Omega)$ to the solution $u$ of the {\em homogenized} 
problem  
\begin{equation}
	\label{HOMOG}
	\displaystyle
-div\left(b\left(\nabla u\right)\right)=f \text{ on $\Omega$, $u\in W_{0}^{1,p_{1}}(\Omega)$},
\end{equation}
where the monotone map $b:\mathbb{R}^{n}\rightarrow\mathbb{R}^{n}$ (independent of $f$ and $\Omega$) can be obtained 
by solving an auxiliary problem for the operator (\ref{01}) on a periodicity cell.

The notion of homogenization is intimately tied to the $\Gamma$-convergence of a suitable family of energy functionals 
$I_{\epsilon}$ as $\epsilon\rightarrow0$ \cite{DalMaso1993}, \cite{Zhikov1994}.  Here the connection is natural in that 
the family of boundary value problems (\ref{HOMOG}) correspond to the Euler equations of the associated energy functionals 
$I_{\epsilon}$ and the solutions $u_{\epsilon}$ are their minimizers.  The homogenized solution is precisely the minimizer 
of the $\Gamma$-limit of the sequence $\{I_{\epsilon}\}_{\epsilon>0}$.  The connections between $\Gamma$ limits and 
homogenization for the power-law materials studied here can be found in \cite{Zhikov1994}.  The explicit formula for the 
$\Gamma$-limit of the associated energy functionals for layered materials was obtained recently in \cite{Pedregal2006}. 

Homogenization theory relates the average behavior seen at large length scales to the underlying heterogeneous structure.  
It allows one to approximate $\{\nabla u_{\epsilon}\}_{\epsilon>0}$ in terms of $\nabla u$, where $u$ is the solution of 
the homogenized problem (\ref{HOMOG}).  The homogenization result given in \cite{Zhikov1994} shows that the average of 
the error incurred in this approximation of $\nabla u_{\epsilon}$ decays to $0$.  

On the other hand it is well known \cite{Kelly1986} that the presence of large local fields either electric or mechanical 
often precede the onset of material failure. For composite materials the presence of  the heterogeneity can amplify the 
applied load and generate local fields with very high intensities. The goal of the analysis presented here is to develop 
tools for quantifying the effect of load transfer between length scales inside heterogeneous media. In this article we 
provide methods for quantitatively measuring the excursions of local fields generated by applied loads. We present a new 
corrector result that delivers  an approximation to $\nabla u_{\epsilon}$ up to an error that converges to zero strongly 
in the norm. Our approach delivers strong approximations for the gradients inside each phase, see, Section~\ref{CorrectorSection}.

The strong approximations are used to develop new tools that provide lower bounds on the local gradient field intensity 
inside micro-structured media.  The bounds are expressed in terms of the $L^q$  norms of gradients of the solutions of 
the local corrector problems.  These results provide a lower bound on the amplification of the macroscopic (average) gradient 
field by the microstructure.  The bounds are shown to hold for every $q$ for which the gradient of the corrector is $L^q$ 
integrable see, Section~\ref{SectionFluctuations}.  The critical values of $q$ for which these moments diverge provide lower 
bounds on the $L^q$ integrability of the gradients $\nabla u_{\epsilon}$ when $\epsilon$ is sufficiently small.  In 
\cite{Lipton2006}, similar lower bounds are established for field concentrations for mixtures of linear electrical conductors 
in the context of two scale convergence. 
	
The corrector results are presented for layered materials and for dispersions of inclusions embedded inside a host medium.  
For the dispersed microstructures the included material is taken to have the lower power-law exponent than that of the host 
phase.  For both of these cases it is shown that the homogenized solution lies in $W_{0}^{1,p_{2}}(\Omega)$.  We use this 
higher order integrability to provide an algorithm for building correctors and construct a sequence of strong approximations 
to the gradients inside each material, see Theorem \ref{corrector}.  When the host phase has a lower power-law exponent than 
the included phase one can only conclude that the homogenized solution lies in $W_{0}^{1,p_{1}}(\Omega)$ and the techniques 
developed here do not apply. 		
	
The earlier work of \cite{DalMaso1990} provides the corrector theory for homogenization of monotone operators that in our 
case applies to composite materials made from constituents having the same power-law growth but with rough coefficients 
$\sigma(x)$.  The corrector theory for monotone operators with uniform power law growth is developed further in 
\cite{Efendiev2004}, where it is used to extend multiscale finite element methods to nonlinear equations for stationary 
random media.  Recent work considers the homogenization of $p_{\epsilon}(x)$-Laplacian boundary value problems for smooth 
exponential functions $p_{\epsilon}(x)$ uniformly converging to a limit function $p_{0}(x)$ \cite{Piatnitski2008}.  There 
the convergence of the family of solutions for these homogenization problems is expressed in the topology of 
$L^{p_{0}(\cdot)}(\Omega)$ \cite{Piatnitski2008}.	
	
The paper is organized as follows. In Section~2, we state the problem and formulate the main results.  Section~3 contains 
the proof of the properties of the homogenized operator.  Section~4 is devoted to proving the higher order integrability of 
the homogenized solution.  Section~5 contains lemmas and  integral inequalities for the correctors used to prove the main 
results.  Section~6 contains the proof of the main results. 
	
\section{Statement of the Problem and Main Results}

\subsection{Notation}	

In this paper we consider two nonlinear power-law materials periodically distributed inside a domain $\Omega\subset\mathbb{R}^{n}$.  
The periodic mixture is described as follows.  We introduce the unit period cell $Y=(0,1)^{n}$ of the microstructure.  Let $F$ 
be an open subset of $Y$ of material one, with smooth boundary $\partial F$, such that $\overline{F}\subset Y$.  The function 
$\chi_{1}(y)=1$ inside $F$ and $0$ outside and $\chi_{2}(y)=1-\chi_{1}(y)$.  We extend $\chi_{1}(y)$ and $\chi_{2}(y)$ by 
periodicity to $\mathbb{R}^{n}$ and the $\epsilon$-periodic mixture inside $\Omega$ is described by the oscillatory characteristic 
functions $\chi_{1}^{\epsilon}(x)=\chi_{1}(x/\epsilon)$ and $\chi_{2}^{\epsilon}(x)=\chi_{2}(x/\epsilon)$.  Here we will consider 
the case where $F$ is given by a simply connected inclusion embedded inside a  host material (see Figure~\ref{fig:disperse}).  A 
distribution of such inclusions is commonly referred to as a periodic dispersion of inclusions.
		
\begin{figure}[h]
	\centering
  \psfrag{F}[c]{$F$}
  \psfrag{Y}[c]{$Y$}
	\includegraphics[width=0.16\textwidth]{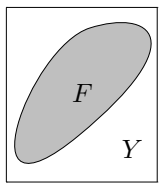}
	\caption{Unit cell: Dispersed Microstructure}
	\label{fig:disperse}
\end{figure}

In this article we also consider layered materials.  For this case the representative unit cell consists of a layer of material one, 
denoted by $R_{1}$, sandwiched between layers of material two, denoted by $R_{2}$.  The interior boundary of $R_{1}$ is denoted by 
$\Gamma$.  Here $\chi_{1}(y)=1$ for $y\in R_{1}$ and $0$ in $R_{2}$, and $\chi_{2}(y)=1-\chi_{1}(y)$ (see Figure~\ref{fig:layer}).
 
\begin{figure}[h]
	\centering
  \psfrag{R1}[c]{$R_2$}
  \psfrag{R2}[c]{$R_1$}
  \psfrag{R3}[c]{$R_2$}
  \psfrag{G}[c]{$\Gamma$}
	\includegraphics[width=0.16\textwidth]{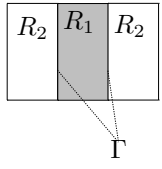}
	\caption{Unit cell: Layered material}
	\label{fig:layer}
\end{figure}

On the unit cell $Y$, the constitutive law for the nonlinear material is given by (\ref{A}) with exponents $p_{1}$ and $p_{2}$ 
satisfying $2\leq p_{1}\leq p_{2}$.  Their H\"{o}lder conjugates are denoted by $q_{2}=p_{1}/(p_{1}-1)$ and $q_{1}=p_{2}/(p_{2}-1)$ 
respectively.  For $i=1,2$, $W_{per}^{1,p_{i}}(Y)$ denotes the set of all functions $u\in W^{1,p_{i}}(Y)$ with mean value zero 
that have the same trace on the opposite faces of $Y$.  Each function $u\in W_{per}^{1,p_{i}}(Y)$ can be extended by periodicity 
to a function of $W_{loc}^{1,p_{i}}(\mathbb{R}^{n})$.

The Euclidean norm and the scalar product in $\mathbb{R}^{n}$ are denoted by $\left|\cdot\right|$ and $\left(\cdot,\cdot\right)$, 
respectively.  If $A\subset\mathbb{R}^{n}$, $\left|A\right|$ denotes the Lebesgue measure and $\chi_{A}(x)$ denotes its 
characteristic function.

The constitutive law for the $\epsilon$-periodic composite is described by $A_{\epsilon}(x,\xi)=A\left(x/\epsilon,\xi\right)$, 
for every $\epsilon>0$, for every $x\in\Omega$, and for every $\xi\in\mathbb{R}^{n}$.

A calculation shows \cite{Bystrom2005} that there exist constants $C_{1},C_{2}>0$ such that for 
almost every $x\in\mathbb{R}^{n}$ and for every $\xi\in\mathbb{R}^{n}$, $A$ satisfies the following 
\begin{enumerate}
\item For all $\xi\in\mathbb{R}^{n}$, $A(\cdot,\xi)$ is $Y$-periodic and Lebesgue measurable.
\item $\left|A(y,0)\right|=0$ for all $y\in\mathbb{R}^{n}$.
\item Continuity
\begin{align}
	\label{ConA}
	\displaystyle \left|A(y,\xi_{1})-A(y,\xi_{2})\right| 	
	&\leq C_{1}\left[\chi_{1}(y)\left|\xi_{1}-\xi_{2}\right|(1+\left|\xi_{1}\right|+\left|\xi_{2}\right|)^{p_{1}-2}\right.\notag\\
	&\quad + \left.\chi_{2}(y)\left|\xi_{1}-\xi_{2}\right|(1+\left|\xi_{1}\right|+\left|\xi_{2}\right|)^{p_{2}-2}\right] 		
\end{align}  
\item Monotonicity  
\begin{align}
	\label{MonA}
	\displaystyle
	\left(A(y,\xi_{1})-A(y,\xi_{2}),\xi_{1}-\xi_{2}\right)\geq 		C_{2}\left(\chi_{1}(y)\left|\xi_{1}-\xi_{2}\right|^{p_{1}}+\chi_{2}(y)\left|\xi_{1}-\xi_{2}\right|^{p_{2}}\right)
\end{align}
\end{enumerate}

\subsection{Dirichlet Boundary Value Problem}

We shall consider the following Dirichlet boundary value problem
\begin{equation}
	\label{Dirichlet}
	\begin{cases}
		-div\left(A_{\epsilon}\left(x,\nabla u_{\epsilon}\right)\right)=f \text{ on $\Omega$},\\
		u_{\epsilon}\in W_{0}^{1,p_{1}}(\Omega);  	
	\end{cases}
\end{equation}
where $f\in W^{-1,q_{2}}(\Omega)$.
 
The following homogenization result holds.

\begin{theorem}[Homogenization Theorem (see \cite{Zhikov1994})]
\label{homogenization}
As $\epsilon\rightarrow0$, the solutions $u_{\epsilon}$ of (\ref{Dirichlet}) converge weakly to $u$ 
in $W^{1,p_{1}}(\Omega)$, where $u$ is the solution of 
\begin{equation}
	\label{homogenized}
	\displaystyle -div\left(b\left(\nabla u\right)\right)=f \text{ on $\Omega$},
\end{equation}	
\begin{equation}
	\label{u}
	\displaystyle u \in W_{0}^{1,p_{1}}(\Omega);  	
\end{equation}
and the function $b:\mathbb{R}^{n}\rightarrow\mathbb{R}^{n}$ is defined for all $\xi\in\mathbb{R}^{n}$ by
\begin{equation}
	\label{b}
	\displaystyle b(\xi)=\int_{Y}A(y,p(y,\xi))dy,
\end{equation}
where $p:\mathbb{R}^{n}\times\mathbb{R}^{n}\rightarrow\mathbb{R}^{n}$ is defined by 
\begin{equation}
	\label{p}
   p(y,\xi)=\xi+\nabla\upsilon_{\xi}(y), 
\end{equation}
where $\upsilon_{\xi}$ is the solution to the cell problem:
\begin{equation}
	\label{cell}
	\begin{cases}
		\displaystyle
		\int_{Y}\left(A(y,\xi+\nabla\upsilon_{\xi}),\nabla w\right)dy=0 \text{, for every $w\in W_{per}^{1,p_{1}}(Y)$},\\
		\upsilon_{\xi}\in W_{per}^{1,p_{1}}(Y)	
	\end{cases}
\end{equation}	
\end{theorem}

\begin{remark}
The following a priori bound is satisfied 
\begin{equation}
	\label{aprioribound}
	\displaystyle
	\sup_{\epsilon>0}\left(\int_{\Omega}\chi_{1}^{\epsilon}(x)\left|\nabla u_{\epsilon}(x)\right|^{p_{1}}dx+\int_{\Omega}\chi_{2}^{\epsilon}(x)\left|\nabla u_{\epsilon}(x)\right|^{p_{2}}dx\right)\leq C<\infty,
\end{equation}
where $C$ does not depend on $\epsilon$.  The proof of this bound is given in Lemma~\ref{proofaprioribound}.
\end{remark}

\begin{remark}
The function $b$, defined in (\ref{b}), satisfies the following properties for every $\xi_{1},\xi_{2}\in\mathbb{R}^{n}$
\begin{enumerate}
\item Continuity: There exists a positive constant $\overline{C_{1}}$ such that
\begin{align}
	\label{Conb}
	\left|b(\xi_{1}) - b(\xi_{2})\right|&\leq \overline{C_{1}}\left[\left|\xi_{1}-\xi_{2}\right|^{\frac{1}{p_{1}-1}}\left(1+\left|\xi_{1}\right|^{p_{1}}+\left|\xi_{2}\right|^{p_{1}}+\left|\xi_{1}\right|^{p_{2}}+\left|\xi_{2}\right|^{p_{2}}\right)^{\frac{p_{1}-2}{p_{1}-1}}\right.\notag\\
		&\quad +\left. \left|\xi_{1}-\xi_{2}\right|^{\frac{1}{p_{2}-1}}\left(1+\left|\xi_{1}\right|^{p_{1}}+\left|\xi_{2}\right|^{p_{1}}+\left|\xi_{1}\right|^{p_{2}}+\left|\xi_{2}\right|^{p_{2}}\right)^{\frac{p_{2}-2}{p_{2}-1}}\right]		
\end{align} 
\item Monotonicity: There exists a positive constant $\overline{C_{2}}$ such that
\begin{align}
	\label{Monb}
	\displaystyle &\left(b(\xi_{1})-b(\xi_{2}),\xi_{1}-\xi_{2}\right)\notag\\
	&\quad \geq \overline{C_{2}}\left(\int_{Y}\chi_{1}(y)\left|p(y,\xi_{1})-p(y,\xi_{2})\right|^{p_{1}}dy+\int_{Y}\chi_{2}(y)\left|p(y,\xi_{1})-p(y,\xi_{2})\right|^{p_{2}}dy\right)\notag\\
	&\quad\geq0
\end{align} 
\end{enumerate}

Properties (\ref{Conb}) and (\ref{Monb}) are proved in Section~\ref{propertiesofb}.
\end{remark}

\begin{remark}	
Since the solution $\upsilon_{\xi}$ of (\ref{cell}) can be extended by periodicity to a function of 
$W_{loc}^{1,p_{1}}(\mathbb{R}^{n})$, then (\ref{cell}) is equivalent to $-div(A(y,\xi+\nabla \upsilon_{\xi}(y)))=0$ 
over $\textit{D}^{'}(\mathbb{R}^{n})$, i.e., 
\begin{equation}
	\label{div with p}
	-div\left(A(y,p(y,\xi))\right)=0 \text{ in $\textsl{D}^{'}(\mathbb{R}^{n})$ for every $\xi\in\mathbb{R}^{n}$}.
\end{equation}

Moreover, by (\ref{cell}), we have
\begin{equation}
	\label{inner product a with p}
	\displaystyle
	\int_{Y}\left(A(y,p(y,\xi)),p(y,\xi)\right)dy=\int_{Y}\left(A(y,p(y,\xi)),\xi\right)dy=\left(b(\xi),\xi\right).
\end{equation}
\end{remark}

For $\epsilon>0$, define $p_{\epsilon}:\mathbb{R}^{n}\times\mathbb{R}^{n}\rightarrow\mathbb{R}^{n}$ by
\begin{equation}
	\label{p epsilon}
	\displaystyle p_{\epsilon}(x,\xi)=p\left(\frac{x}{\epsilon},\xi\right)=\xi+\nabla\upsilon_{\xi}\left(\frac{x}{\epsilon}\right),
\end{equation}
where $\upsilon_{\xi}$ is the unique solution of (\ref{cell}).  The functions $p$ and $p_{\epsilon}$ are easily 
seen to have the following properties
\begin{equation}
	\label{p1}
	\text{$p(\cdot,\xi)$ is $Y$-periodic and $p_{\epsilon}(x,\xi)$ is $\epsilon$-periodic in $x$.}
\end{equation}
\begin{equation}
	\label{p2}
	\displaystyle \int_{Y}p(y,\xi)dy=\xi.
\end{equation}
\begin{equation}
	\label{p3}
	p_{\epsilon}(\cdot,\xi)\rightharpoonup\xi \text{ in $L^{p_{1}}(\Omega;\mathbb{R}^{n})$ as $\epsilon\rightarrow0$.}
\end{equation}
\begin{equation}
	\label{p4}
	p(y,0)=0 \text{ for almost every $y$.}
\end{equation}
\begin{equation}
	\label{p5}
	A\left(\frac{\cdot}{\epsilon},p_{\epsilon}(\cdot,\xi)\right)\rightharpoonup b(\xi) \text{ in $L^{q_{2}}(\Omega;\mathbb{R}^{n})$, as $\epsilon\rightarrow 0$}.
\end{equation}

We now state the higher order integrability properties of the homogenized solution for periodic dispersions of inclusions 
and layered microgeometries.

\begin{theorem}
\label{regularity}
Given a periodic dispersion of inclusions or a layered material then the solution $u$ of (\ref{homogenized}) belongs to 
$W_{0}^{1,p_{2}}(\Omega)$.
\end{theorem}

The proof of this theorem is given in Section~\ref{regularity of u}.

\subsubsection{Statement of the Corrector Theorem}
\label{CorrectorSection}

We now describe the family of correctors that provide a strong approximation of the sequence $\left\{\chi_{i}^{\epsilon}\nabla u_{\epsilon}\right\}_{\epsilon>0}$  in the $L^{p_{i}}(\Omega,\mathbb{R}^{n})$ norm.  We denote the rescaled period cell 
with side length $\epsilon>0$  by $Y_{\epsilon}$ and write $Y_{\epsilon}^{i}=\epsilon i+Y_{\epsilon}$, where $i\in\mathbb{Z}^{n}$.  
In what follows it is convenient to define the index set $I_{\epsilon}=\left\{i\in\mathbb{Z}^{n}:Y_{\epsilon}^{i}\subset \Omega \right\}$.
For $\varphi\in L^{p_{2}}(\Omega;\mathbb{R}^{n})$, we define the local average operator $M_{\epsilon}$ associated with the 
partition $Y^i_\epsilon$, ${i\in I_{\epsilon}}$ by
\begin{equation}
	\label{approximation}
	\displaystyle M_{\epsilon}(\varphi)(x) = 
	\begin{cases}
		\displaystyle
\sum_{i\in\hspace{1mm}I_{\epsilon}}{\chi_{Y_{\epsilon}^{i}}(x)\frac{1}{\left|Y_{\epsilon}^{i}\right|}\int_{Y_{\epsilon}^{i}}\varphi(y)dy}; & \text{ if $\displaystyle x\in\bigcup_{i\in I_{\epsilon}}Y_{\epsilon}^{i}$,}\\
		0; & \text{if $\displaystyle x\in\Omega\setminus\bigcup_{i\in I_{\epsilon}}Y_{\epsilon}^{i}$.}
	\end{cases}
\end{equation}

The family $M_{\epsilon}$ has the following properties
\begin{enumerate}
	\item For $i=1,2$, $\left\|M_{\epsilon}(\varphi)-\varphi\right\|_{L^{p_{i}}(\Omega;\mathbb{R}^{n})}\rightarrow0$ 
	as $\epsilon\rightarrow0$ (see \cite{Zaanen1958}).
	\item $M_{\epsilon}(\varphi)\rightarrow\varphi$ a.e. on $\Omega$ (see \cite{Zaanen1958}).
	\item From Jensen's inequality we have   $\left\|M_{\epsilon}(\varphi)\right\|_{L^{p_{i}}(\Omega;\mathbb{R}^{n})}\leq\left\|\varphi\right\|_{L^{p_{i}}(\Omega;\mathbb{R}^{n})}$, 
	for every $\varphi\in L^{p_{2}}(\Omega;\mathbb{R}^{n})$ and $i=1,2$.
\end{enumerate}

The strong approximation to the sequence $\left\{\chi_{i}^{\epsilon}\nabla u_{\epsilon}\right\}_{\epsilon>0}$ 
is given by the following corrector theorem. 

\begin{theorem}[Corrector Theorem]
\label{corrector}
Let $f\in W^{-1,q_{2}}(\Omega)$, let $u_{\epsilon}$ be the solutions to the problem (\ref{Dirichlet}), 
and let $u$ be the solution to problem (\ref{homogenized}).  Then, for periodic dispersions of inclusions  and 
for layered materials, we have 
\begin{equation}
 	\label{strongcvcorrector}	
 	\displaystyle
 	\int_{\Omega}\left|\chi_{i}^{\epsilon}(x)p_{\epsilon}\left(x,M_{\epsilon}(\nabla u)(x)\right)-\chi_{i}^{\epsilon}(x)\nabla u_{\epsilon}(x)\right|^{p_{i}}dx\rightarrow0, 
\end{equation}
as $\epsilon\rightarrow0$, for $i=1,2$.
\end{theorem}

The proof of Theorem~\ref{corrector} is given in Section~\ref{proof corrector}. 

\subsubsection{Lower Bounds on the Local Amplification of the Macroscopic Field}
\label{SectionFluctuations}

We display lower bounds on the $L^q$ norm of the gradient fields inside each material that are given in terms 
of the correctors presented in Theorem \ref{corrector}.  We begin by presenting a general lower bound that holds 
for the composition of the sequence $\{\chi_i^\epsilon\nabla u_\epsilon\}_{\epsilon>0}$ with any non-negative 
Carath\'{e}odory function. Recall that $\psi:\Omega\times\mathbb{R}^{n}\rightarrow\mathbb{R}$ is a Carath\'{e}odory 
function if $\psi(x,\cdot)$ is continuous for almost every $x\in\Omega$ and if $\psi(\cdot,\lambda)$ is measurable 
in $x$ for every $\lambda\in\mathbb{R}^n$.  The lower bound on the sequence obtained by the composition of 
$\psi(x,\cdot)$ with $\chi_i^\epsilon(x)\nabla u_\epsilon(x)$ is given by

\begin{theorem}
\label{fluctuations}
For all Carath\'{e}odory functions $\psi\geq0$ and measurable sets $D\subset\Omega$ we have 
$$\int_{D}\int_{Y}\psi\left(x,\chi_{i}(y)p\left(y,\nabla u(x)\right)\right)dydx\leq\liminf_{\epsilon\rightarrow0}\int_{D}\psi\left(x,\chi_{i}^{\epsilon}(x)\nabla u_{\epsilon}(x)\right)dx.$$
	
If the sequence $\left\{\psi\left(x,\chi_{i}^{\epsilon}(x)\nabla u_{\epsilon}(x)\right)\right\}_{\epsilon>0}$ 
is weakly convergent in $L^{1}(\Omega)$, then the inequality becomes an equality.
	
In particular, for $\psi(x,\lambda)=\left|\lambda\right|^{q}$ with $q\geq 2$, we have
\begin{eqnarray} \int_{D}\int_{Y}\chi_{i}(y)\left|p\left(y,\nabla u(x)\right)\right|^{q}dydx\leq\liminf_{\epsilon\rightarrow0}\int_{D}\chi_{i}^{\epsilon}(x)\left|\nabla u_{\epsilon}(x)\right|^{q}dx.
\label{lq}
\end{eqnarray}
\end{theorem}

Theorem \ref{fluctuations} together with \eqref{lq} provide explicit lower bounds on the gradient field inside
each material.  It relates the local excursions of the gradient inside each phase $\chi_i^\epsilon\nabla u_\epsilon$
to the average gradient $\nabla u$ through the multiscale quantity given by the corrector $p(y,\nabla u(x))$.  It is 
clear from \eqref{lq} that the $L^q(Y\times\Omega)$ integrability of $p(y,\nabla u(x))$ provides a lower bound on the 
$L^q(\Omega)$ integrability of $\nabla u_\epsilon$.

The proof of Theorem~\ref{fluctuations} is given in Section~\ref{proof fluctuations}. 
	
\section{Properties of the Homogenized Operator $b$}
\label{propertiesofb}

In this section, we prove properties (\ref{Conb}) and (\ref{Monb}) of the homogenized operator $b$.  In the rest 
of the paper, the letter $C$ will represent a generic positive constant independent of $\epsilon$, and it can take 
different values.  

\subsection{Proof of (\ref{Monb})}

Using (\ref{cell}) and (\ref{MonA}), we have
\begin{align*}
	\displaystyle 
	&\left(b(\xi_{2}) - b(\xi_{1}),\xi_{2}-\xi_{1}\right)
	 = \int_{Y}\left(A(y,p(y,\xi_{2}))-A(y,p(y,\xi_{1})),p(y,\xi_{2})-p(y,\xi_{1})\right)dy\\	
	&\quad \geq C\left(\int_{Y}\chi_{1}(y)\left|p(y,\xi_{1})-p(y,\xi_{2})\right|^{p_{1}}dy+\int_{Y}\chi_{2}(y)\left|p(y,\xi_{1})-p(y,\xi_{2})\right|^{p_{2}}dy\right)\\
	&\quad\geq0. 
\end{align*}	

\subsection{Proof of (\ref{Conb})}

By (\ref{ConA}), H\"older's inequality, and (\ref{MonA}) we have
\begin{align}
\label{conb1}
	\displaystyle 
	\left|b(\xi_{1}) - b(\xi_{2})\right|&\leq \int_{Y}\left|A(y,p(y,\xi_{1}))-A(y,p(y,\xi_{2}))\right|dy\\
	&\leq C\left(\int_{Y}\chi_{1}(y)\left|p(y,\xi_{1})-p(y,\xi_{2})\right|^{p_{1}}dy\right)^{\frac{1}{p_{1}}}\notag\\
	&\qquad\times\left(\int_{Y}\chi_{1}(y)(1+\left|p(y,\xi_{1})\right|+\left|p(y,\xi_{2})\right|)^{q_{2}(p_{1}-2)}dy\right)^{\frac{1}{q_{2}}}\notag \\
	&\quad+C\left(\int_{Y}\chi_{2}(y)\left|p(y,\xi_{1})-p(y,\xi_{2})\right|^{p_{2}}dy\right)^{\frac{1}{p_{2}}}
\notag\\
	&\qquad\times\left(\int_{Y}\chi_{2}(y)(1+\left|p(y,\xi_{1})\right|+\left|p(y,\xi_{2})\right|)^{q_{1}(p_{2}-2)}dy\right)^{\frac{1}{q_{1}}}\notag \\
	&\leq C\left[\int_{Y}\left(A(y,p(y,\xi_{1}))-A(y,p(y,\xi_{2})),p(y,\xi_{1})-p(y,\xi_{2})\right)dy\right]^{\frac{1}{p_{1}}}\notag\\
	&\qquad\times\left[\int_{Y}\chi_{1}(y)(1+\left|p(y,\xi_{1})\right|+\left|p(y,\xi_{2})\right|)^{q_{2}(p_{1}-2)}dy\right]^{\frac{1}{q_{2}}}\notag\\
	&\quad+C\left[\int_{Y}\left(A(y,p(y,\xi_{1}))-A(y,p(y,\xi_{2})),p(y,\xi_{1})-p(y,\xi_{2})\right)dy\right]^{\frac{1}{p_{2}}}\notag\\
	&\qquad\times\left[\int_{Y}\chi_{2}(y)(1+\left|p(y,\xi_{1})\right|+\left|p(y,\xi_{2})\right|)^{q_{1}(p_{2}-2)}dy\right]^{\frac{1}{q_{1}}}\notag
	&\intertext{Using (\ref{conb1}), (\ref{cell}), (\ref{b}), the Cauchy-Schwarz inequality, Lemma~\ref{lemma1}, and Young's inequality we obtain}\notag
	&\leq C\left[\left(\frac{\delta^{p_{1}}}{p_{1}}+\frac{\delta^{p_{2}}}{p_{2}}\right)\left|b(\xi_{1})-b(\xi_{2})\right|
\right.\notag\\
&\quad+\frac{\delta^{-q_{2}}\left|\xi_{1}-\xi_{2}\right|^{\frac{1}{p_{1}-1}}\left(1+\left|\xi_{1}\right|^{p_{1}}+\left|\xi_{2}\right|^{p_{1}}+\left|\xi_{1}\right|^{p_{2}}+\left|\xi_{2}\right|^{p_{2}}\right)^{\frac{p_{1}-2}{p_{1}-1}}}{q_{2}}\notag\\
&\quad\left.+\frac{\delta^{-q_{1}}\left|\xi_{1}-\xi_{2}\right|^{\frac{1}{p_{2}-1}}\left(1+\left|\xi_{1}\right|^{p_{1}}+\left|\xi_{2}\right|^{p_{1}}+\left|\xi_{1}\right|^{p_{2}}+\left|\xi_{2}\right|^{p_{2}}\right)^{\frac{p_{2}-2}{p_{2}-1}}}{q_{1}}\right]\notag
\end{align}

Rearranging the terms in (\ref{conb1}), and taking $\delta$ small enough we obtain (\ref{Conb})

\section{Higher Order Integrability of the Homogenized Solution}
\label{regularity of u}	

In this section we display higher integrability results for the field gradients inside dispersed microstructures 
and layered materials.  For dispersions of inclusions, the included material is taken to have a lower power-law 
exponent than that of the host phase.  For both of these cases it is shown that the homogenized solution lies in 
$W_{0}^{1,p_{2}}(\Omega)$.  In the following sections we will apply these facts to establish strong approximations 
for the sequences $\{\chi_i^{\epsilon}\nabla u_{\epsilon}\}_{\epsilon>0}$ in $L^{p_2}(\Omega,{\mathbb{R}}^n)$.  
The approach taken here is variational and uses the \textit{homogenized Lagrangian} associated with $b(\xi)$ 
defined in (\ref{b}).  The integrability of the homogenized solution $u$ of (\ref{homogenized}) is determined by 
the growth of the homogenized Lagrangian with respect to its argument.  

To proceed we introduce the local Lagrangian associated with power-law composites.  The Lagrangian corresponding 
to the problem studied here is given by
\begin{equation}
	\label{Lagrangian}
\tilde{f}(x,\xi)=q(x)\left|\xi\right|^{p(x)} \text{, with $q(x)=\frac{\sigma_{1}}{p_{1}}\chi_{1}(x)+\frac{\sigma_{2}}{p_{2}}\chi_{2}(x),$}
\end{equation}
where $\xi\in\mathbb{R}^{n}$ and $x\in\Omega\subset\mathbb{R}^{n}$.  Here $\nabla_{\xi}\tilde{f}(x,\xi)=A\left(x,\xi\right)$, 
where $A(x,\xi)$ is given by (\ref{A}).

We consider the rescaled Lagrangian
\begin{equation}
	\label{epsilonLagrangian}
\tilde{f_{\epsilon}}(x,\xi)=\tilde{f}\left(\frac{x}{\epsilon},\xi\right)=\frac{\sigma_{1}}{p_{1}}\chi_{1}^{\epsilon}(x)\left|\xi\right|^{p_{1}}+\frac{\sigma_{2}}{p_{2}}\chi_{2}^{\epsilon}(x)\left|\xi\right|^{p_{2}},
\end{equation}
where $\chi_{i}^{\epsilon}(x)=\chi_{i}\left(x/\epsilon\right)$, $i=1,2$, $\xi\in\mathbb{R}^{n}$, 
and $x\in\Omega\subset\mathbb{R}^{n}$.  

The Dirichlet problem given by (\ref{Dirichlet}) is associated with the variational problem given by 
\begin{equation}
	\displaystyle
	\label{energy1}
	\displaystyle E_{1}^{\epsilon}(f)=\inf_{u\in W_{0}^{1,p_{1}}(\Omega)}\left\{\int_{\Omega}\tilde{f}_{\epsilon}(x,\nabla u)dx-\left\langle f,u\right\rangle\right\},
\end{equation}	
with $f\in W^{-1,q_{2}}(\Omega)$.  Here (\ref{Dirichlet}) is the Euler equation for (\ref{energy1}).  
However, we also consider
\begin{equation}
	\displaystyle
	\label{energy2}
	\displaystyle E_{2}^{\epsilon}(f)=\inf_{u\in W_{0}^{1,p_{2}}(\Omega)}\left\{\int_{\Omega}\tilde{f}_{\epsilon}(x,\nabla u)dx-\left\langle f,u\right\rangle\right\},
\end{equation}
with $f\in W^{-1,q_{2}}(\Omega)$ (See \cite{Zhikov1992}).  Here $\left\langle\cdot,\cdot\right\rangle$ is the duality 
pairing between $W_{0}^{1,p_{1}}(\Omega)$ and $W^{-1, q_{2}}(\Omega)$.

From \cite{Zhikov1994}, we have $\displaystyle \lim_{\epsilon\rightarrow0}E^{\epsilon}_{i}=E_{i}$, for $i=1,2$, where
\begin{equation}
	\label{Ei}
	\displaystyle E_{i}=\inf_{u\in W_{0}^{1,p_{i}}(\Omega)}\left\{\int_{\Omega}\hat{\tilde{f}}_{i}(\nabla u(x))dx-\left\langle f,u\right\rangle\right\}.
\end{equation}	

In (\ref{Ei}), $\hat{\tilde{f}}_{i}(\xi)$ is given by 
\begin{equation}
	\displaystyle
	\label{fi}
	\hat{\tilde{f}}_{i}(\xi)=\inf_{\text{$v$ in $W_{per}^{1,p_{i}}(Y)$}}\int_{Y}\tilde{f}(y,\xi+\nabla v(y))dy
\end{equation}
and satisfies
\begin{equation}
	\label{nonstandard}
	-c_{0}+c_{1}\left|\xi\right|^{p_{1}}\leq \hat{\tilde{f}}_{i}(\xi) \leq c_{2}\left|\xi\right|^{p_{2}}+c_{0}.
\end{equation}

In general, (see \cite{Zhikov1995}) Lavrentiev phenomenon can occur and $E_{1}<E_{2}$.  However, for periodic 
dispersed and layered microstructures, no Lavrentiev phenomenon occurs and we have the following 
Homogenization Theorem.

\begin{theorem}
\label{REG}
For periodic dispersed and layered 
microstructures, the homogenized Dirichlet problems satisfy $E_{1}=E_{2}$, where 
$\hat{\tilde{f}}=\hat{\tilde{f}}_{1}=\hat{\tilde{f}}_{2}$ and $c_{2}+c_{1}\left|\xi\right|^{p_{2}}\leq\hat{\tilde{f}}(\xi)$.  
Moreover, $\nabla_{\xi}\hat{\tilde{f}}(\xi)=b(\xi)$, where $b$ is the homogenized operator (\ref{b}). 
\end{theorem}

\begin{proof}
Theorem~\ref{REG} has been proved for dispersed periodic media in \cite{Zhikov1994}.  We prove Theorem~\ref{REG} 
for layers following the steps outlined in \cite{Zhikov1994}.

We first show that $\hat{\tilde{f}}=\hat{\tilde{f}}_{1}=\hat{\tilde{f}}_{2}$ holds for layered media.  Then we 
show that the homogenized Lagrangian $\hat{\tilde{f}}$ satisfies the estimate given by 
\begin{equation}
	\label{standard}
	\displaystyle	-c_{0}+c_{1}\left|\xi\right|^{p_{2}}\leq \hat{\tilde{f}}(\xi) \leq c_{2}\left|\xi\right|^{p_{2}}+c_{0}
\end{equation}
with $c_{0}\geq0$, and $c_{1}$,$c_{2}>0$.

We introduce the space of functions $W^{1,p_{2}}_*(R_{2})$ that belong to $W^{1,p_{2}}(R_{2})$ and are periodic 
on $\partial R_2\cap\partial Y$.

\begin{lemma}
\label{extensionlayers}
Any function in $v\in W^{1,p_{2}}_*(R_{2})$ can be extended to $R_{1}$ in such a 
way that the extension $\tilde{v}(y)$ belongs to $W_{per}^{1,p_{2}}(Y)$ and $\tilde{v}(y)=v(y)$ on $R_{2}$.
\end{lemma}

\begin{proof}
Let $\varphi$ to be the solution of 
\begin{equation*}
	\begin{cases}
		\Delta_{p_{2}}\varphi=0 & \text{, on $R_{1}$}\\
		\varphi \text{ takes periodic boundary values on opposite faces of $\partial Y\cap\partial R_{1}$}\\
		\varphi_{\big|_{1}}=v_{\big|_{2}} & \text{, on $\Gamma$}
	\end{cases}
\end{equation*}	
Here the subscript $1$ indicates the trace on the $R_{1}$ side of $\Gamma$ and $2$ indicates the trace 
on the $R_{2}$ side of $\Gamma$.  For a proof of existence of the solution $\varphi$ see \cite{Evans1982} or 
\cite{Lewis1977}.  

The extension $\tilde{v}$ is given by 
$$\displaystyle \tilde{v}=\begin{cases}	v & \text{, in $R_{2}.$}\\	\varphi & \text{, on $R_{1}.$}	\end{cases}$$
\end{proof}

To prove that $\hat{\tilde{f}}_{1}=\hat{\tilde{f}}_{2}$, it suffices to show that for every $v\in W_{per}^{1,p_{1}}(Y)$ 
satisfying  $\displaystyle \int_{Y}\tilde{f}(y,\xi+\nabla v(y))dy<\infty$   there exists a sequence $v_{\epsilon}\in W_{per}^{1,p_{2}}(Y)$ such that 
$$\lim_{\epsilon\rightarrow0}\int_{Y}\tilde{f}(y,\xi+\nabla v_{\epsilon}(y))dy= \int_{Y}\tilde{f}(y,\xi+\nabla v(y))dy.$$

For $v$ as above, let $\tilde{v}$ be as 
in Lemma~\ref{extensionlayers} and set $z=v-\tilde{v}$. It is clear that $z\in W^{1,p_{1}}(R_{1})$, is periodic on opposite faces of 
$\partial Y\cap\partial R_{1}$, zero on $\Gamma$ and we write
$$\int_{Y}\tilde{f}(y,\xi+\nabla v(y))dy=\int_{R_{2}}f_{2}(\xi+\nabla v(y))dy+\int_{R_{1}}f_{1}(\xi+\nabla \tilde{v}(y)+\nabla z(y))dy,$$
where $f_{1}(\xi)=\frac{\sigma_{1}}{p_{1}}\left|\xi\right|^{p_1}$ and $f_{2}(\xi)=\frac{\sigma_{2}}{p_{2}}\left|\xi\right|^{p_2}$. 

We can choose a sequence $\left\{z_{\epsilon}\right\}_{\epsilon>0}\in\textsl{C}_{0}^{\infty}(R_{1})$ such that $z_{\epsilon}$
vanishes in $R_{2}$ and $z_{\epsilon}\rightarrow z$ in $W^{1,p_{1}}(R_{1})$.
	
Define $v_{\epsilon}\in W_{per}^{1,p_{2}}(Y)$ by 
\begin{equation*}
	\displaystyle v_{\epsilon}=
	\begin{cases}
		v & \text{in $R_{2}$},\\
		\tilde{v}+z_{\epsilon} & \text{in $R_{1}$}.  	
	\end{cases}
\end{equation*}
	
Since $v_{\epsilon}\rightarrow v$ in $W_{per}^{1,p_{1}}(Y)$, we see that 
\begin{align*}
	\displaystyle 
	&\lim_{\epsilon\rightarrow 0}\int_{Y}\tilde{f}(y,\xi+\nabla v_{\epsilon}(y))dy\\
	&\quad=\lim_{\epsilon\rightarrow 0}\left(\int_{R_{2}}f_{2}(\xi+\nabla v(y))dy+\int_{R_{1}}f_{1}(\xi+\nabla \tilde{v}(y)+\nabla z_{\epsilon}(y))dy\right)\\
	&\quad=\int_{Y}\tilde{f}(y,\xi+\nabla v(y))dy.
\end{align*}	
Therefore $\hat{\tilde{f}}=\hat{\tilde{f}}_{1}=\hat{\tilde{f}}_{2}$ for layered media.
	
We establish (\ref{standard}) by introducing the convex conjugate of $\hat{\tilde{f}}$.  We denote the convex 
dual of $\hat{\tilde{f}}_{i}(\xi)$ by $\hat{\tilde{g}}_{i}(\xi)$; i.e., $\displaystyle\hat{\tilde{g}}_{i}(\xi)=\sup_{\lambda\in\mathbb{R}^{n}}\left\{\xi\cdot\lambda-\hat{\tilde{f}}_{i}(\lambda)\right\}$.  
It is easily verified (see \cite{Zhikov1992}) that
\begin{equation}
	\displaystyle
	\label{gi}
	\hat{\tilde{g}}_{i}(\xi)=\inf_{\text{$w$ in $Sol^{q_{i}}(Y)$}}\int_{Y}\tilde{g}(y,\xi+w(y))dy
\end{equation}
and 
\begin{equation}
	\label{dualnonstandard}
	-c_{0}+c_{1}^{*}\left|\xi\right|^{q_{1}}\leq \hat{\tilde{g}}_{i}(\xi) \leq c_{2}^{*}\left|\xi\right|^{q_{2}}+c_{0}.
\end{equation}
Here $Sol^{q_{i}}(Y)$ are the solenoidal vector fields belonging to $L^{q_{i}}(Y,\mathbb{R}^{n})$ and having mean 
value zero $$Sol^{q_{i}}(Y)=\left\{w\in L^{q_{i}}(Y;\mathbb{R}^{n}):\text{div}\,w=0, w\cdot n \text{ anti-periodic}\right\}.$$

We will show that $\hat{\tilde{g}}=\hat{\tilde{g}}_{1}=\hat{\tilde{g}}_{2}$ satisfies 
$\hat{\tilde{g}}(\xi) \leq c_{2}\left|\xi\right|^{q_{1}}+c_{1},$ and apply duality to recover 
$\hat{\tilde{f}}(\xi) \geq c_{2}^{*}\left|\xi\right|^{p_{2}}+c_{1}^{*}.$
		
To get the upper bound on $\hat{\tilde{g}}$ we use the following lemma.
\begin{lemma}
\label{lemmatau}
There exists $\tau$ with $\text{div}\,\tau=0$ in $Y$, such that $\tau\cdot n$ is anti-periodic on the boundary 
of $Y$, $\tau=-\xi$ in $R_{1}$, and $$\int_{Y}\left|\tau(y)\right|^{q_{1}}dy\leq C\left|\xi\right|^{q_{1}}.$$
\end{lemma}

\begin{proof}
Let the function $\varphi\in W^{1,p_{2}}_*(R_{2})$ be the solution of
\begin{equation*}
	\begin{cases}
		\hbox{ $\nabla\varphi|\nabla\varphi|^{p-2}\cdot n$ is anti-periodic on $\partial R_2\cap\partial Y$};\\
		\Delta_{p_{2}}\varphi=0  \text{ in $R_{2}$;}\\
		\nabla\varphi\left|\nabla\varphi\right|^{p_{2}-2}\cdot n_{\big|_{2}}=-\xi\cdot n_{\big|_{1}};  \text{ on $\Gamma$,}
	\end{cases}
\end{equation*}
where the subscript $1$ indicates the trace on the $R_{1}$ side of $\Gamma$ and $2$ indicates the trace on the 
$R_{2}$ side of $\Gamma$.  The Neumann problem given above is the stationarity condition for the energy 
$\displaystyle \int_{R_2}|\nabla\phi|^{p_2}dx-\int_\Gamma\phi \xi\cdot n\,dS$ when minimized over all 
$\phi\in W_*^{1,p_2}(R_2)$.  The solution of the Neumann problem is unique up to a constant.  Here the anti-periodic 
boundary condition on $\nabla\varphi|\nabla\varphi|^{p-2}\cdot n$ is the natural boundary condition for the problem.
	
Now we define $\tau$ according to  
\begin{equation*}
	\displaystyle
	\tau=\begin{cases}
		-\xi;&\text{ in $R_{1}$}\\ 
		\nabla\varphi\left|\nabla\varphi\right|^{p_{2}-2};&\text{ in $R_{2}$}
	\end{cases}
\end{equation*}
and it follows that
\begin{equation}
	\label{above}
	\left|\tau\right|^{q_{1}}=
	\begin{cases}
		\left|\xi\right|^{q_{1}}; & \text{ in $R_{1}$}\\
\left[\left(\nabla\varphi\left|\nabla\varphi\right|^{p_{2}-2}\right)^{2}\right]^{\frac{q_{1}}{2}}=\left(\left|\nabla\varphi\right|^{p_{2}-1}\right)^{q_{1}}=\left|\nabla\varphi\right|^{p_{2}}; & \text{ in $R_{2}$}.
	\end{cases}
\end{equation}	
	 
Then, for $\psi\in W^{1,p_{2}}_*(R_{2})$ we have
\begin{align}
	\label{above1}
	&\int_{R_{2}}\left|\nabla\varphi\right|^{p_{2}-2}\nabla\varphi\cdot\nabla\psi dy\\
	&\quad=\int_{\Gamma}\psi\left|\nabla\varphi\right|^{p_{2}-2}\nabla\varphi\cdot ndS+\int_{\partial R_{2}\cap\partial Y}\psi\left|\nabla\varphi\right|^{p_{2}-2}\nabla\varphi\cdot ndS \notag\\
	&\quad=-\int_{\Gamma}\psi\xi\cdot ndS=-\int_{R_2}\nabla\psi\cdot\xi\,dy.\notag
\end{align} 
	
Set $\psi=\varphi$ in (\ref{above1}) and an application of  H\"{o}lder's inequality gives
\begin{equation}
	\label{above2}
	\displaystyle
	\int_{R_{2}}\left|\nabla\varphi(y)\right|^{p_{2}}dy\leq\int_{R_{2}}\left|\xi\right|^{q_{1}}dy.
\end{equation}
	
Therefore, using (\ref{above}) and (\ref{above2}), we have 
\begin{align*}
	\displaystyle
	\int_{Y}\left|\tau(y)\right|^{q_{1}}dy&=\int_{R_{1}}\left|\tau(y)\right|^{q_{1}}dy+\int_{R_{2}}\left|\tau(y)\right|^{q_{1}}dy\\
	&=\int_{R_{1}}\left|\xi\right|^{q_{1}}dy+\int_{R_{2}}\left|\nabla\varphi(y)\right|^{p_{2}}dy\leq C\left|\xi\right|^{q_{1}}.
\end{align*}
\end{proof}

Taking $\hat{\tilde{g}}$ to be the conjugate of $\hat{\tilde{f}}$, and choosing $\tau$ in $Sol^{q_{1}}(Y)$ 
as in Lemma~\ref{lemmatau}, we obtain 
\begin{align*}
	&\hat{\tilde{g}}(\xi)=\inf_{\text{$\tau$ in $Sol^{q_{1}}(Y)$}}\int_{Y}\tilde{g}(y,\xi+\tau)dy\leq\int_{Y}\tilde{g}(y,\xi+\tau)dy\\
	&\quad\leq \int_{R_{1}}\tilde{g}(y,0)dy+\int_{R_{2}}\tilde{g}(y,\xi+\tau)dy\leq c_{1}+c_{2}\int_{R_{2}}\left|\xi+\tau\right|^{q_{1}}dy\leq c_{1}+c_{2}\left|\xi\right|^{q_{1}},	
\end{align*}
and the left hand inequality in (\ref{standard}) follows from duality.

This concludes the proof of Theorem~\ref{REG}.
\end{proof}

Collecting results we now prove Theorem~\ref{regularity}.  Indeed the minimizer of $E_{1}$ is precisely 
the solution $u$ of (\ref{homogenized}) and (\ref{u}).  Theorem~\ref{REG} establishes the coercivity of 
$E_{1}$ over $W_{0}^{1,p_{2}}(\Omega)$, thus the solution $u$ lies in $W_{0}^{1,p_{2}}(\Omega)$.

\section{Some Useful Lemmas and Estimates}
	
In this section we state and prove a priori bounds and convergence properties for the sequences $p_{\epsilon}$ 
defined in (\ref{p epsilon}), $\nabla u_{\epsilon}$, and $A_{\epsilon}(x, p_{\epsilon}(x,\nabla u_{\epsilon}))$ 
that are used in the proof of the main results of this paper.  
	
\begin{lemma}
\label{lemma1}
For every $\xi\in\mathbb{R}^{n}$ we have
\begin{equation}
	\label{Lemma 1}
	\displaystyle 
	\int_{Y}\chi_{1}(y)\left|p(y,\xi)\right|^{p_{1}}dy + \int_{Y}\chi_{2}(y)\left|p(y,\xi)\right|^{p_{2}}dy \leq C\left(1+\left|\xi\right|^{p_{1}}\theta_{1}+\left|\xi\right|^{p_{2}}\theta_{2}\right),
\end{equation}
and by a change of variables, we obtain
\begin{equation}
	\label{Lemma 1 epsilon}
	\displaystyle 
	\int_{Y_{\epsilon}}\chi_{1}^{\epsilon}(x)\left|p_{\epsilon}(x,\xi)\right|^{p_{1}}dx + \int_{Y_{\epsilon}}\chi_{2}^{\epsilon}(x)\left|p_{\epsilon}(x,\xi)\right|^{p_{2}}dx \leq C\left(1+\left|\xi\right|^{p_{1}}\theta_{1}+\left|\xi\right|^{p_{2}}\theta_{2}\right)\left|Y_{\epsilon}\right|
\end{equation}
\end{lemma}

\begin{proof}
Let $\xi\in\mathbb{R}^{n}$.  By (\ref{MonA}) we have that 
$$\left(A(y,p(y,\xi)),p(y,\xi)\right)\geq C\left(\chi_{1}(y)\left|p(y,\xi)\right|^{p_{1}}+\chi_{2}(y)\left|p(y,\xi)\right|^{p_{2}}\right)$$

Integrating both sides over $Y$, using (\ref{ConA}), and Young's Inequality, we get
\begin{align*}
	\displaystyle
	&\int_{Y}\chi_{1}(y)\left|p(y,\xi)\right|^{p_{1}}dy + \int_{Y}\chi_{2}(y)\left|p(y,\xi)\right|^{p_{2}}dy\\
	&\quad\leq C\left[\left(\delta^{q_{2}}\theta_{1}+\delta^{q_{1}}\theta_{2}\right)+\left(\frac{\left|\xi\right|^{p_{1}}\theta_{1}}{\delta^{p_{1}}}+\frac{\left|\xi\right|^{p_{2}}\theta_{2}}{\delta^{p_{2}}}\right) \right.\\
	&\qquad + \left. (\delta^{q_{2}}+\delta^{q_{1}})\left(\int_{Y}\chi_{1}(y)\left|p(y,\xi)\right|^{p_{1}}dy+\int_{Y}\chi_{2}(y)\left|p(y,\xi)\right|^{p_{2}}dy\right)\right]
\end{align*}
		
Doing some algebraic manipulations, we obtain 
\begin{align*}
	\displaystyle
	&\left(1-C(\delta^{q_{2}}+\delta^{q_{1}})\right)\left(\int_{Y}\chi_{1}(y)\left|p(y,\xi)\right|^{p_{1}}dy + \int_{Y}\chi_{2}(y)\left|p(y,\xi)\right|^{p_{2}}dy\right)\\
	&\quad\leq C\left[\left(\delta^{q_{2}}\theta_{1}+\delta^{q_{1}}\theta_{2}\right)+\left(\delta^{-p_{1}}\left|\xi\right|^{p_{1}}\theta_{1}+\delta^{-p_{2}}\left|\xi\right|^{p_{2}}\theta_{2}\right)\right]
\end{align*}
		
On choosing an appropiate $\delta$, we finally obtain (\ref{Lemma 1}).
\end{proof}

\begin{lemma}
\label{lemma2}
For every $\xi_{1},\xi_{2}\in\mathbb{R}^{n}$ we have	
\begin{align}
	\label{Lemma 2}
	\displaystyle
	&\int_{Y}\chi_{1}(y)\left|p(y,\xi_{1})-p(y,\xi_{2})\right|^{p_{1}}dy+\int_{Y}\chi_{2}(y)\left|p(y,\xi_{1})-p(y,\xi_{2})\right|^{p_{2}}dy\\	
	&\quad\leq C\left[\left(1+\left|\xi_{1}\right|^{p_{1}}\theta_{1}+\left|\xi_{1}\right|^{p_{2}}\theta_{2}+\left|\xi_{2}\right|^{p_{1}}\theta_{1}+\left|\xi_{2}\right|^{p_{2}}\theta_{2}\right)^{\frac{p_{1}-2}{p_{1}-1}}\left|\xi_{1}-\xi_{2}\right|^{\frac{p_{1}}{p_{1}-1}}\theta_{1}^{\frac{1}{p_{1}-1}}\right.\notag\\
	&\qquad\left.+ \left(1+\left|\xi_{1}\right|^{p_{1}}\theta_{1}+\left|\xi_{1}\right|^{p_{2}}\theta_{2}+\left|\xi_{2}\right|^{p_{1}}\theta_{1}+\left|\xi_{2}\right|^{p_{2}}\theta_{2}\right)^{\frac{p_{2}-2}{p_{2}-1}}\left|\xi_{1}-\xi_{2}\right|^{\frac{p_{2}}{p_{2}-1}}\theta_{2}^{\frac{1}{p_{2}-1}}\right]\notag
\end{align}		
and by doing a change a variables, we obtain
\begin{align}
	\label{Lemma 2 epsilon}
	\displaystyle 
	&\int_{Y_{\epsilon}}\chi_{1}^{\epsilon}(x)\left|p_{\epsilon}(x,\xi_{1})-p_{\epsilon}(x,\xi_{2})\right|^{p_{1}}dx + \int_{Y_{\epsilon}}\chi_{2}^{\epsilon}(x)\left|p_{\epsilon}(x,\xi_{1})-p_{\epsilon}(x,\xi_{2})\right|^{p_{2}}dx\\
	&\quad\leq C\left[\left(1+\left|\xi_{1}\right|^{p_{1}}\theta_{1}+\left|\xi_{1}\right|^{p_{2}}\theta_{2}+\left|\xi_{2}\right|^{p_{1}}\theta_{1}+\left|\xi_{2}\right|^{p_{2}}\theta_{2}\right)^{\frac{p_{1}-2}{p_{1}-1}}\left|\xi_{1}-\xi_{2}\right|^{\frac{p_{1}}{p_{1}-1}}\theta_{1}^{\frac{1}{p_{1}-1}}\right.\notag\\
	&\qquad\left.+ \left(1+\left|\xi_{1}\right|^{p_{1}}\theta_{1}+\left|\xi_{1}\right|^{p_{2}}\theta_{2}+\left|\xi_{2}\right|^{p_{1}}\theta_{1}+\left|\xi_{2}\right|^{p_{2}}\theta_{2}\right)^{\frac{p_{2}-2}{p_{2}-1}}\left|\xi_{1}-\xi_{2}\right|^{\frac{p_{2}}{p_{2}-1}}\theta_{2}^{\frac{1}{p_{2}-1}}\right]\left|Y_{\epsilon}\right|\notag 
\end{align}
\end{lemma}

\begin{proof}
By (\ref{MonA}), (\ref{cell}), and (\ref{ConA}) we have that
\begin{align*}
\label{lemma2-1}
	\displaystyle
	&\int_{Y}\chi_{1}(y)\left|p(y,\xi_{1})-p(y,\xi_{2})\right|^{p_{1}}dy+\int_{Y}\chi_{2}(y)\left|p(y,\xi_{1})-p(y,\xi_{2})\right|^{p_{2}}dy\\
	&\quad\leq C\int_{Y}\left|A(y,p(y,\xi_{1}))-A(y,p(y,\xi_{2}))\right|\left|\xi_{1}-\xi_{2}\right|dy\\
  &\quad\leq C\left[\int_{Y}\chi_{1}(y)\left|p(y,\xi_{1})-p(y,\xi_{2})\right|\left(1+\left|p(y,\xi_{1})\right|+\left|p(y,\xi_{2})\right|\right)^{p_{1}-2}\left|\xi_{1}-\xi_{2}\right|dy\right.\\
  &\qquad+\left. \int_{Y}\chi_{2}(y)\left|p(y,\xi_{1})-p(y,\xi_{2})\right|\left(1+\left|p(y,\xi_{1})\right|+\left|p(y,\xi_{2})\right|\right)^{p_{2}-2}\left|\xi_{1}-\xi_{2}\right|dy\right] 
	&\intertext{Using Holder's inequality in the first term with $\displaystyle r_{1}=p_{1}/(p_{1}-2)$, $r_{2}=p_{1}$, $r_{3}=p_{1}$, and in the second term with $\displaystyle s_{1}=p_{2}/(p_{2}-2)$, $s_{2}=p_{2}$, $s_{3}=p_{2}$, and using Lemma~\ref{lemma1}, we obtain}
	&\quad\leq C\left[\left(1+\left|\xi_{1}\right|^{p_{1}}\theta_{1}+\left|\xi_{1}\right|^{p_{2}}\theta_{2}+\left|\xi_{2}\right|^{p_{1}}\theta_{1}+\left|\xi_{2}\right|^{p_{2}}\theta_{2}\right)^{\frac{p_{1}-2}{p_{1}}}\right.\\
	&\qquad\quad\times			\left|\xi_{1}-\xi_{2}\right|\theta_{1}^{\frac{1}{p_{1}}}\left(\int_{Y}\chi_{1}(y)\left|p(y,\xi_{1})-p(y,\xi_{2})\right|^{p_{1}}dy\right)^{\frac{1}{p_{1}}}\\
	&\qquad+ \left(1+\left|\xi_{1}\right|^{p_{1}}\theta_{1}+\left|\xi_{1}\right|^{p_{2}}\theta_{2}+\left|\xi_{2}\right|^{p_{1}}\theta_{1}+\left|\xi_{2}\right|^{p_{2}}\theta_{2}\right)^{\frac{p_{2}-2}{p_{2}}}\\
	&\qquad\quad\times\left.			\left|\xi_{1}-\xi_{2}\right|\theta_{2}^{\frac{1}{p_{2}}}\left(\int_{Y}\chi_{2}(y)\left|p(y,\xi_{1})-p(y,\xi_{2})\right|^{p_{2}}dy\right)^{\frac{1}{p_{2}}}\right]
	&\intertext{By Young's inequality, we get}
	&\quad\leq C\left[\frac{\delta^{-q_{2}}\left(1+\left|\xi_{1}\right|^{p_{1}}\theta_{1}+\left|\xi_{1}\right|^{p_{2}}\theta_{2}+\left|\xi_{2}\right|^{p_{1}}\theta_{1}+\left|\xi_{2}\right|^{p_{2}}\theta_{2}\right)^{\frac{(p_{1}-2)q_{2}}{p_{1}}}\left|\xi_{1}-\xi_{2}\right|^{q_{2}}\theta_{1}^{\frac{q_{2}}{p_{1}}}}{q_{2}}\right.\\
	&\qquad+ \frac{\delta^{p_{1}}\displaystyle\int_{Y}\chi_{1}(y)\left|p(y,\xi_{1})-p(y,\xi_{2})\right|^{p_{1}}dy}{p_{1}}+\frac{\delta^{p_{2}}\displaystyle\int_{Y}\chi_{2}(y)\left|p(y,\xi_{1})-p(y,\xi_{2})\right|^{p_{2}}dy}{p_{2}}\\
	&\qquad+\left. \frac{\delta^{-q_{1}}\left(1+\left|\xi_{1}\right|^{p_{1}}\theta_{1}+\left|\xi_{1}\right|^{p_{2}}\theta_{2}+\left|\xi_{2}\right|^{p_{1}}\theta_{1}+\left|\xi_{2}\right|^{p_{2}}\theta_{2}\right)^{\frac{(p_{2}-2)q_{1}}{p_{2}}}\left|\xi_{1}-\xi_{2}\right|^{q_{1}}\theta_{2}^{\frac{q_{1}}{p_{2}}}}{q_{1}} \right]
\end{align*}
	
Straightforward algebraic manipulation delivers 
\begin{align*}
\displaystyle
	& k_{\delta}\left(\int_{Y}\chi_{1}(y)\left|p(y,\xi_{1})-p(y,\xi_{2})\right|^{p_{1}}dy+\int_{Y}\chi_{2}(y)\left|p(y,\xi_{1})-p(y,\xi_{2})\right|^{p_{2}}dy\right)\\
	&\quad\leq  C\left[\frac{\delta^{-q_{2}}\left(1+\left|\xi_{1}\right|^{p_{1}}\theta_{1}+\left|\xi_{1}\right|^{p_{2}}\theta_{2}+\left|\xi_{2}\right|^{p_{1}}\theta_{1}+\left|\xi_{2}\right|^{p_{2}}\theta_{2}\right)^{\frac{p_{1}-2}{p_{1}-1}}\left|\xi_{1}-\xi_{2}\right|^{\frac{p_{1}}{p_{1}-1}}\theta_{1}^{\frac{1}{p_{1}-1}}}{q_{2}}\right.\\
	&\qquad+\left. \frac{\delta^{-q_{1}}\left(1+\left|\xi_{1}\right|^{p_{1}}\theta_{1}+\left|\xi_{1}\right|^{p_{2}}\theta_{2}+\left|\xi_{2}\right|^{p_{1}}\theta_{1}+\left|\xi_{2}\right|^{p_{2}}\theta_{2}\right)^{\frac{p_{2}-2}{p_{2}-1}}\left|\xi_{1}-\xi_{2}\right|^{\frac{p_{2}}{p_{2}-1}}\theta_{2}^{\frac{1}{p_{2}-1}}}{q_{1}}\right] 
\end{align*}
where $k_{\delta}=\min\left\{\left(1-\frac{C\delta^{p_{1}}}{p_{1}}\right),\left(1-\frac{C\delta^{p_{2}}}{p_{2}}\right)\right\}$.

The result follows on choosing $\delta$ small enough so that $k_{\delta}$ is positive. 
\end{proof}

\begin{lemma}
\label{lemma3} 
Let $\varphi$ be such that $$\sup_{\epsilon>0}\left\{\int_{\Omega}\chi_{1}^{\epsilon}(x)\left|\varphi(x)\right|^{p_{1}}dx + \int_{\Omega}\chi_{2}^{\epsilon}(x)\left|\varphi(x)\right|^{p_{2}}dx\right\}<\infty, $$
and let $\Psi$ be a simple function of the form 
\begin{equation}
	\label{Psi}
	\Psi(x)=\sum_{j=0}^{m}\eta_{j}\chi_{\Omega_{j}}(x),
\end{equation}	
with $\eta_{j}\in\mathbb{R}^{n}\setminus\left\{0\right\}$, $\Omega_{j}\subset\subset\Omega$, 
$\left|\partial\Omega_{j}\right|=0$, $\Omega_{j}\cap\Omega_{k}=\emptyset$ for $j\neq k$ and $j,k=1,...,m$; 
and set $\eta_{0}=0$ and $\displaystyle \Omega_{0}=\Omega\setminus\bigcup_{j=1}^{m}\Omega_{j}$.	 Then
\begin{align}
\label{lemma3formula}
	\displaystyle 		&\limsup_{\epsilon\rightarrow0}\left(\int_{\Omega}\chi_{1}^{\epsilon}(x)\left|p_{\epsilon}(x,M_{\epsilon}\varphi(x))-p_{\epsilon}(x,\Psi(x))\right|^{p_{1}}dx\right.\\
	&\quad+\left.\int_{\Omega}\chi_{2}^{\epsilon}(x)\left|p_{\epsilon}(x,M_{\epsilon}\varphi(x))-p_{\epsilon}(x,\Psi(x))\right|^{p_{2}}dx\right)\notag\\
	&\qquad\leq\limsup_{\epsilon\rightarrow0}\,C\sum_{i=1}^{2}\left[\left(\left|\Omega\right|+\int_{\Omega}\chi_{1}^{\epsilon}(x)\left|\varphi(x)\right|^{p_{1}}dx + \int_{\Omega}\chi_{2}^{\epsilon}(x)\left|\varphi(x)\right|^{p_{2}}dx\notag\right.\right.\\
	&\qquad\quad+\left. \int_{\Omega}\chi_{1}^{\epsilon}(x)\left|\Psi(x)\right|^{p_{1}}dx+\int_{\Omega}\chi_{2}^{\epsilon}(x)\left|\Psi(x)\right|^{p_{2}}dx\right)^{\frac{p_{i}-2}{p_{i}-1}}\notag\\
	&\qquad\qquad\times \left. \left(\int_{\Omega}\chi_{i}^{\epsilon}(x)\left|\varphi(x)-\Psi(x)\right|^{p_{i}}dx\right)^{\frac{1}{p_{i}-1}}\right]\notag
\end{align}
\end{lemma}
	
\begin{proof}
Let $\Psi$ of the form (\ref{Psi}).  For every $\epsilon>0$, let us denote by 
$\displaystyle\Omega_{\epsilon}=\bigcup_{i\in I_{\epsilon}}\overline{Y_{\epsilon}^{i}};$ 
and for $j=0,1,2,...,m$, we set $$I_{\epsilon}^{j}=\left\{i\in I_{\epsilon}:Y_{\epsilon}^{i}\subseteq\Omega_{j}\right\}\text{, and } J_{\epsilon}^{j}=\left\{i\in I_{\epsilon}:Y_{\epsilon}^{i}\cap\Omega_{j}\neq\emptyset, Y_{\epsilon}^{i}\setminus\Omega_{j}\neq\emptyset\right\}.$$

Furthermore, $\displaystyle E_{\epsilon}^{j}=\bigcup_{i\in I_{\epsilon}^{j}} \overline{Y_{\epsilon}^{i}}$, 
$\displaystyle F_{\epsilon}^{j}=\bigcup_{i\in J_{\epsilon}^{j}} \overline{Y_{\epsilon}^{i}}$, and 
as $\epsilon\rightarrow0$, we have $\left|F_{\epsilon}^{j}\right|\rightarrow0$. 
		
Set $$\displaystyle \xi_{\epsilon}^{i}=\frac{1}{\left|Y_{\epsilon}^{i}\right|}\int_{Y_{\epsilon}^{i}}\varphi(y)dy.$$
For $\epsilon$ sufficiently small $\Omega_{j}$ ($j\neq0$) is contained in $\Omega_{\epsilon}$.

From (\ref{Psi}), (\ref{approximation}), using the fact that $\Omega_{j}\subset E_{\epsilon}^{j}\cup F_{\epsilon}^{j}$, 
Lemma~\ref{lemma2}, and H\"older's inequality it follows that
\begin{align}
\label{lemma3-1}
	\displaystyle
	&\int_{\Omega}\chi_{1}^{\epsilon}(x)\left|p_{\epsilon}(x,M_{\epsilon}\varphi)-p_{\epsilon}(x,\Psi)\right|^{p_{1}}dx + \int_{\Omega}\chi_{2}^{\epsilon}(x)\left|p_{\epsilon}(x,M_{\epsilon}\varphi)-p_{\epsilon}(x,\Psi)\right|^{p_{2}}dx\notag\\
	& \leq
C\left[\left(\left|\Omega\right|+\int_{\Omega}\chi_{1}^{\epsilon}(x)\left|M_{\epsilon}\varphi-\varphi\right|^{p_{1}}dx+\int_{\Omega}\chi_{1}^{\epsilon}(x)\left|\varphi\right|^{p_{1}}dx+\int_{\Omega}\chi_{2}^{\epsilon}(x)\left|M_{\epsilon}\varphi-\varphi\right|^{p_{2}}dx \right.\right.\notag\\
	&+\left.
\int_{\Omega}\chi_{2}^{\epsilon}(x)\left|\varphi(x)\right|^{p_{2}}dx + \int_{\Omega}\chi_{1}^{\epsilon}(x)\left|\Psi(x)\right|^{p_{1}}dx + \int_{\Omega}\chi_{2}^{\epsilon}(x)\left|\Psi(x)\right|^{p_{2}}dx\right)^{\frac{p_{1}-2}{p_{1}-1}}\notag\\
	&\quad\times
\left(\int_{\Omega}\chi_{1}^{\epsilon}(x)\left|M_{\epsilon}\varphi-\varphi\right|^{p_{1}}dx + \int_{\Omega}\chi_{1}^{\epsilon}(x)\left|\varphi-\Psi\right|^{p_{1}}dx\right)^{\frac{1}{p_{1}-1}}\notag\\
	&+\left(\left|\Omega\right|+\int_{\Omega}\chi_{1}^{\epsilon}(x)\left|M_{\epsilon}\varphi-\varphi\right|^{p_{1}}dx + \int_{\Omega}\chi_{1}^{\epsilon}(x)\left|\varphi\right|^{p_{1}}dx + \int_{\Omega}\chi_{2}^{\epsilon}(x)\left|M_{\epsilon}\varphi-\varphi\right|^{p_{2}}dx \right.\notag\\
	&+ \left.
\int_{\Omega}\chi_{2}^{\epsilon}(x)\left|\varphi(x)\right|^{p_{2}}dx+\int_{\Omega}\chi_{1}^{\epsilon}(x)\left|\Psi(x)\right|^{p_{1}}dx+\int_{\Omega}\chi_{2}^{\epsilon}(x)\left|\Psi(x)\right|^{p_{2}}dx\right)^{\frac{p_{2}-2}{p_{2}-1}}\notag\\
	&\quad\times \left.
\left(\int_{\Omega}\chi_{2}^{\epsilon}(x)\left|M_{\epsilon}\varphi-\varphi\right|^{p_{2}}dx+\int_{\Omega}\chi_{2}^{\epsilon}(x)\left|\varphi-\Psi\right|^{p_{2}}dx\right)^{\frac{1}{p_{2}-1}}\right]\notag\\
	& + 
C\sum_{j=0}^{m}\left[\left(\left|F_{\epsilon}^{j}\right| + \int_{F_{\epsilon}^{j}}\left|M_{\epsilon}\varphi(x)\right|^{p_{1}}\theta_{1}dx + \int_{F_{\epsilon}^{j}}\left|M_{\epsilon}\varphi(x)\right|^{p_{2}}\theta_{2}dx\right.\right.\notag\\
	& +\left. 
\left|\eta_{j}\right|^{p_{1}}\theta_{1}\left|F_{\epsilon}^{j}\right| + \left|\eta_{j}\right|^{p_{2}}\theta_{2}\left|F_{\epsilon}^{j}\right|\right)^{\frac{p_{1}-2}{p_{1}-1}}\left(\int_{F_{\epsilon}^{j}}\theta_{1}\left|\sum_{i\in J_{\epsilon}^{j}}\chi_{Y_{\epsilon}^{i}}(x)\xi_{\epsilon}^{i}-\eta_{j}\right|^{p_{1}}dx\right)^{\frac{1}{p_{1}-1}}\notag\\		
	& + 
\left(\left|F_{\epsilon}^{j}\right|+\int_{F_{\epsilon}^{j}}\left|M_{\epsilon}\varphi(x)\right|^{p_{1}}\theta_{1}dx+\int_{F_{\epsilon}^{j}}\left|M_{\epsilon}\varphi(x)\right|^{p_{2}}\theta_{2}dx\right.\notag\\
	& + \left. \left|\eta_{j}\right|^{p_{1}}\theta_{1}\left|F_{\epsilon}^{j}\right|+\left|\eta_{j}\right|^{p_{2}}\theta_{2}\left|F_{\epsilon}^{j}\right|\Big)^{\frac{p_{2}-2}{p_{2}-1}}\left(\int_{F_{\epsilon}^{j}}\theta_{2}\left|\sum_{i\in J_{\epsilon}^{j}}\chi_{Y_{\epsilon}^{i}}(x)\xi_{\epsilon}^{i}-\eta_{j}\right|^{p_{2}}dx\right)^{\frac{1}{p_{2}-1}}\right]
\end{align}

Since $\left|\partial\Omega_{j}\right|=0$ for $j\neq0$, we have that $\left|F_{\epsilon}^{j}\right|\rightarrow0$ 
as $\epsilon\rightarrow0$, for every $j=0,1,2,...,m$.
	
By Property~(1) of $M_{\epsilon}$ mentioned in Section~\ref{CorrectorSection}, we have $$\displaystyle\int_{\Omega}\chi_{i}^{\epsilon}(x)\left|M_{\epsilon}\varphi(x)-\varphi(x)\right|^{p_{i}}dx\rightarrow0, \text{ as $\epsilon\rightarrow0$, for $i=1,2$.}$$ 
	
Therefore, taking $\limsup$ as $\epsilon\rightarrow0$ in (\ref{lemma3-1}), we obtain (\ref{lemma3formula}).
\end{proof}

\begin{lemma}
\label{uniform boundedness of p at M}
	If the microstructure is dispersed or layered, we have that $$\sup_{\epsilon>0}\left\{\int_{\Omega}\chi_{i}^{\epsilon}(x)\left|p_{\epsilon}(x,M_{\epsilon}\nabla u(x))\right|^{p_{i}}dx\right\}\leq C<\infty \text{, for $i=1,2$.}$$
\end{lemma}

\begin{proof}
Using (\ref{approximation}), we have
\begin{align*}
	\displaystyle 
	& 
\int_{\Omega}\chi_{1}^{\epsilon}(x)\left|p_{\epsilon}(x,M_{\epsilon}\nabla u(x))\right|^{p_{1}}dx + \int_{\Omega}\chi_{2}^{\epsilon}(x)\left|p_{\epsilon}(x,M_{\epsilon}\nabla u(x))\right|^{p_{2}}dx\\
	&\quad = \sum_{i\in\textbf{I}_{\epsilon}}\left[\int_{Y_{\epsilon}^{i}}\chi_{1}^{\epsilon}(x)\left|p_{\epsilon}(x,\xi_{\epsilon}^{i})\right|^{p_{1}}dx + \int_{Y_{\epsilon}^{i}}\chi_{2}^{\epsilon}(x)\left|p_{\epsilon}(x,\xi_{\epsilon}^{i})\right|^{p_{2}}dx\right]\\
	&\quad \leq  C\sum_{i\in\textbf{I}_{\epsilon}}\left(1+\left|\xi_{\epsilon}^{i}\right|^{p_{1}}\theta_{1}+\left|\xi_{\epsilon}^{i}\right|^{p_{2}}\theta_{2}\right)\left|Y_{\epsilon}^{i}\right|\\
	&\quad = C\sum_{i\in\textbf{I}_{\epsilon}}\left(\left|Y_{\epsilon}^{i}\right|+\left|\xi_{\epsilon}^{i}\right|^{p_{1}}\theta_{1}\left|Y_{\epsilon}^{i}\right|+\left|\xi_{\epsilon}^{i}\right|^{p_{2}}\theta_{2}\left|Y_{\epsilon}^{i}\right|\right)\\
	&\quad \leq C\left(\left|\Omega\right|+\left\|\nabla u\right\|^{p_{1}}_{\textbf{L}^{p_{1}}(\Omega)}+\left\|\nabla u\right\|^{p_{2}}_{\textbf{L}^{p_{2}}(\Omega)}\right)<\infty,
	\end{align*}
where the last three inequalities follow from Lemma~\ref{lemma1}, Jensen's inequality, and Theorem~\ref{regularity}.
\end{proof}

\begin{lemma}
\label{proofaprioribound}
Let $u_{\epsilon}$ be the solution to (\ref{Dirichlet}).  Then (\ref{aprioribound}) holds.
\end{lemma}

\begin{proof}
Evaluating $u_{\epsilon}$ in the weak formulation for (\ref{Dirichlet}), applying H\"older's inequality, and 
since $f\in W^{-1,q_{2}}(\Omega)$, we obtain 
\begin{align}
	\label{aprioriboundproof1}
	\displaystyle	
	&\int_{\Omega}(A_{\epsilon}(x,\nabla u_{\epsilon}),\nabla u_{\epsilon})dx=\sigma_{1}\int_{\Omega}\chi_{1}^{\epsilon}(x)\left|\nabla u_{\epsilon}\right|^{p_{1}}dx+\sigma_{2}\int_{\Omega}\chi_{2}^{\epsilon}(x)\left|\nabla u_{\epsilon}\right|^{p_{2}}dx\\
	&\quad =\left\langle f, u_{\epsilon}\right\rangle\leq C\left[\left(\int_{\Omega}\chi_{1}^{\epsilon}(x)\left|\nabla u_{\epsilon}\right|^{p_{1}}dx\right)^{\frac{1}{p_{1}}}+\left(\int_{\Omega}\chi_{2}^{\epsilon}(x)\left|\nabla u_{\epsilon}\right|^{p_{2}}dx\right)^{\frac{1}{p_{2}}}\right]\notag
\end{align}

Applying Young's inequality to the last term in (\ref{aprioriboundproof1}), we obtain	
\begin{align}
	\label{aprioriboundproof2}
	\displaystyle
	& \sigma_{1}\int_{\Omega}\chi_{1}^{\epsilon}(x)\left|\nabla u_{\epsilon}\right|^{p_{1}}dx+\sigma_{2}\int_{\Omega}\chi_{2}^{\epsilon}(x)\left|\nabla u_{\epsilon}\right|^{p_{2}}dx\\
	& \quad \leq C\left[\frac{\delta^{p_{1}}}{p_{1}}\int_{\Omega}\chi_{1}^{\epsilon}(x)\left|\nabla u_{\epsilon}\right|^{p_{1}}dx+\frac{\delta^{-q_{2}}}{q_{2}}+\frac{\delta^{p_{2}}}{p_{2}}\int_{\Omega}\chi_{2}^{\epsilon}(x)\left|\nabla u_{\epsilon}\right|^{p_{2}}dx+\frac{\delta^{-q_{1}}}{q_{1}}\right]\notag
\end{align}	

By rearranging the terms in (\ref{aprioriboundproof2}), one gets
\begin{align*}
	\displaystyle
	&\left(\sigma_{1}-C\frac{\delta^{p_{1}}}{p_{1}}\right)\int_{\Omega}\chi_{1}^{\epsilon}(x)\left|\nabla u_{\epsilon}\right|^{p_{1}}dx+\left(\sigma_{2}-C\frac{\delta^{p_{2}}}{p_{2}}\right)\int_{\Omega}\chi_{2}^{\epsilon}(x)\left|\nabla u_{\epsilon}\right|^{p_{2}}dx\\
	&\quad\leq \frac{\delta^{-q_{2}}}{q_{2}}+\frac{\delta^{-q_{1}}}{q_{1}}.
\end{align*}

Therefore, by choosing $\delta$ small enough so that $\min\left\{\sigma_{1}-C\frac{\delta^{p_{1}}}{p_{1}},\sigma_{2}-C\frac{\delta^{p_{2}}}{p_{2}}\right\}$ is positive, one obtains $$\int_{\Omega}\chi_{1}^{\epsilon}(x)\left|\nabla u_{\epsilon}(x)\right|^{p_{1}}dx+\int_{\Omega}\chi_{2}^{\epsilon}(x)\left|\nabla u_{\epsilon}(x)\right|^{p_{2}}dx\leq C.$$
\end{proof}

\begin{lemma}
\label{unifboundDP}
For all $j=0,...,m$, we have that $\displaystyle \int_{\Omega_{j}}\left|\left(A_{\epsilon}\left(x,p_{\epsilon}\left(x,\eta_{j}\right)\right),\nabla u_{\epsilon}(x)\right)\right|dx$ and $\displaystyle \int_{\Omega_{j}}\left|\left(A_{\epsilon}\left(x,\nabla u_{\epsilon}(x)\right),p_{\epsilon}\left(x,\eta_{j}\right)\right)\right|dx$ are uniformly bounded with respect to $\epsilon$.
\end{lemma}

\begin{proof}
Using H\"older's inequality, (\ref{ConA}), and (\ref{aprioribound}), we obtain
\begin{align*}	
	\displaystyle 
	&\int_{\Omega_{j}}\left|\left(A_{\epsilon}\left(x,p_{\epsilon}\left(x,\eta_{j}\right)\right),\nabla u_{\epsilon}(x)\right)\right|dx\leq \int_{\Omega_{j}}\left|A_{\epsilon}\left(x,p_{\epsilon}\left(x,\eta_{j}\right)\right)\right|\left|\nabla u_{\epsilon}(x)\right|dx\\
	&\quad\leq C\left[\left(\int_{\Omega_{j}}\chi_{1}^{\epsilon}(x)\left(1+\left|p_{\epsilon}\left(x,\eta_{j}\right)\right|\right)^{p_{1}}dx\right)^{\frac{1}{q_{2}}}+\left(\int_{\Omega_{j}}\chi_{2}^{\epsilon}(x)\left(1+\left|p_{\epsilon}\left(x,\eta_{j}\right)\right|\right)^{p_{2}}dx\right)^{\frac{1}{q_{1}}}\right]\\
	&\quad\leq C \text{, where $C$ does not depend on $\epsilon$}.	
\end{align*}
The proof of the uniform boundedness of $\displaystyle \int_{\Omega_{j}}\left|\left(A_{\epsilon}\left(x,\nabla u_{\epsilon}(x)\right),p_{\epsilon}\left(x,\eta_{j}\right)\right)\right|dx$ follows in the same manner.
\end{proof}

\begin{lemma}
\label{dunfordpettis}
	As $\epsilon\rightarrow0$, up to a subsequence, $\left(A_{\epsilon}\left(\cdot,p_{\epsilon}\left(\cdot,\eta_{j}\right)\right),\nabla u_{\epsilon}(\cdot)\right)$ converges weakly to a function $g_{j}\in L^{1}(\Omega_{j};\mathbb{R})$, for all $j=0,...,m$.  
In a similar way, up to a subsequence, $\left(A_{\epsilon}\left(\cdot,\nabla u_{\epsilon}(\cdot)\right),p_{\epsilon}\left(\cdot,\eta_{j}\right)\right)$ converges weakly to a function $h_{j}\in L^{1}(\Omega_{j};\mathbb{R})$, for all $j=0,...,m$.
\end{lemma}

\begin{proof}
We prove the first statement of the lemma, the second statement follows in a similar way.  The lemma follows from 
the Dunford-Pettis theorem (see \cite{Dacorogna1989}).  To apply this theorem we establish the following conditions:
\begin{enumerate}
	\item $\displaystyle \int_{\Omega_{j}}\left|\left(A_{\epsilon}\left(x,p_{\epsilon}\left(x,\eta_{j}\right)\right),\nabla u_{\epsilon}(x)\right)\right|dx$ 
	is uniformly bounded with respect to $\epsilon$
	\item For all $j=0,...,m$, $\left(A_{\epsilon}\left(\cdot,p_{\epsilon}\left(\cdot,\eta_{j}\right)\right),\nabla u_{\epsilon}(\cdot)\right)$ 
	is equiintegrable.
\end{enumerate}
	
The first condition is proved in Lemma~\ref{unifboundDP}.  For the second condition, we have that $\displaystyle \chi_{1}^{\epsilon}(\cdot)\left|A_{\epsilon}\left(\cdot,p_{\epsilon}(\cdot,\eta_{j})\right)\right|^{q_{2}}$ and 
$\displaystyle \chi_{2}^{\epsilon}(\cdot)\left|A_{\epsilon}\left(\cdot,p_{\epsilon}(\cdot,\eta_{j})\right)\right|^{q_{1}}$ 
are equiintegrable (see for example \textit{Theorem~1.5} of \cite{Dacorogna1989}).

By (\ref{aprioribound}), for any $E\subset\Omega$, we have $$\displaystyle \max_{i=1,2}\left\{\sup_{\epsilon>0}\left\{\left(\int_{E}\chi_{i}^{\epsilon}(x)\left|\nabla u_{\epsilon}(x)\right|^{p_{i}}dx\right)^{\frac{1}{p_{i}}}\right\}\right\}\leq C.$$
	
Let $\alpha>0$ arbitrary and choose $\alpha_{1}>0$ and $\alpha_{2}>0$ such that $\alpha_{1}^{1/q_{2}}+\alpha_{2}^{1/q_{1}}<\alpha/C.$
	
For $\alpha_{1}$ and $\alpha_{2}$, there exist $\lambda(\alpha_{1})>0$ and $\lambda(\alpha_{2})>0$ such that for 
every $E\subset\Omega$ with $\left|E\right|<\min\left\{\lambda(\alpha_{1}),\lambda(\alpha_{2})\right\}$, $$\int_{E}\chi_{1}^{\epsilon}(x)\left|A_{\epsilon}\left(x,p_{\epsilon}\left(x,\eta_{j}\right)\right)\right|^{q_{2}}dx<\alpha_{1} \text{, and   $\int_{E}\chi_{2}^{\epsilon}(x)\left|A_{\epsilon}\left(x,p_{\epsilon}\left(x,\eta_{j}\right)\right)\right|^{q_{1}}dx<\alpha_{2}$.}$$
	
Take $\lambda=\lambda(\alpha)=\min\left\{\lambda(\alpha_{1}),\lambda(\alpha_{2})\right\}$.  Then, 
for all $E\subset\Omega$ with $\left|E\right|<\lambda(\alpha)$, we have 
\begin{align*}
	\displaystyle 
	&\int_{E}\left|\left(A_{\epsilon}\left(x,p_{\epsilon}\left(x,\eta_{j}\right)\right),\nabla u_{\epsilon}(x)\right)\right|dx\leq \int_{E}\left|A_{\epsilon}\left(x,p_{\epsilon}\left(x,\eta_{j}\right)\right)\right|\left|\nabla u_{\epsilon}(x)\right|dx\\
	&\quad\leq \left(\int_{E}\chi_{1}^{\epsilon}(x)\left|A_{\epsilon}\left(x,p_{\epsilon}\left(x,\eta_{j}\right)\right)\right|^{q_{2}}dx\right)^{\frac{1}{q_{2}}}\left(\int_{E}\chi_{1}^{\epsilon}(x)\left|\nabla u_{\epsilon}(x)\right|^{p_{1}}dx\right)^{\frac{1}{p_{1}}}\\
	&\qquad+ \left(\int_{E}\chi_{2}^{\epsilon}(x)\left|A_{\epsilon}\left(x,p_{\epsilon}\left(x,\eta_{j}\right)\right)\right|^{q_{1}}dx\right)^{\frac{1}{q_{1}}}\left(\int_{E}\chi_{2}^{\epsilon}(x)\left|\nabla u_{\epsilon}(x)\right|^{p_{2}}dx\right)^{\frac{1}{p_{2}}}\\
	&\quad\leq C(\alpha_{1}^{1/q_{2}}+\alpha_{2}^{1/q_{1}}) < \alpha,
\end{align*}
for every $\alpha>0$, and so $\left(A_{\epsilon}\left(\cdot,p_{\epsilon}\left(\cdot,\eta_{j}\right)\right),\nabla u_{\epsilon}(\cdot)\right)$ is equiintegrable. 
\end{proof}

\section{Proof of Main Results}

\subsection{Proof of the Corrector Theorem}	
\label{proof corrector}

We are now in the position to give the proof of Theorem~\ref{corrector}.
\begin{proof}
Let $u_{\epsilon}\in W_{0}^{1,p_{1}}(\Omega)$ the solutions of (\ref{Dirichlet}).  By (\ref{MonA}), we have that
\begin{align*}
	\displaystyle	
	& 
\int_{\Omega}\left[\chi_{1}^{\epsilon}(x)\left|p_{\epsilon}\left(x,M_{\epsilon}\nabla u(x)\right)-\nabla u_{\epsilon}(x)\right|^{p_{1}} + \chi_{2}^{\epsilon}(x)\left|p_{\epsilon}\left(x,M_{\epsilon}\nabla u(x)\right)-\nabla u_{\epsilon}(x)\right|^{p_{2}}\right]dx\\ 
	& \quad\leq
C\int_{\Omega}\left(A_{\epsilon}\left(x,p_{\epsilon}\left(x,M_{\epsilon} \nabla u(x)\right)\right)-A_{\epsilon}\left(x,\nabla u_{\epsilon}(x)\right),p_{\epsilon}\left(x,M_{\epsilon}\nabla u(x)\right)-\nabla u_{\epsilon}(x)\right)dx
\end{align*}

To prove Theorem~\ref{corrector}, we show that 
\begin{align*}
	\displaystyle
	& \int_{\Omega}\left(A_{\epsilon}\left(x,p_{\epsilon}\left(x,M_{\epsilon}\nabla u(x)\right)\right)-A_{\epsilon}\left(x,\nabla u_{\epsilon}(x)\right),p_{\epsilon}\left(x,M_{\epsilon}\nabla u(x)\right)-\nabla u_{\epsilon}(x)\right)dx\\
	&=\int_{\Omega}\left(A_{\epsilon}\left(x,p_{\epsilon}\left(x,M_{\epsilon}\nabla u\right)\right),p_{\epsilon}\left(x,M_{\epsilon}\nabla u\right)\right)dx-\int_{\Omega}\left(A_{\epsilon}\left(x,p_{\epsilon}\left(x,M_{\epsilon}\nabla u\right)\right),\nabla u_{\epsilon}\right)dx\\
	&\quad-\int_{\Omega}\left(A_{\epsilon}\left(x,\nabla u_{\epsilon}\right),p_{\epsilon}\left(x,M_{\epsilon}\nabla u\right)\right)dx+\int_{\Omega}\left(A_{\epsilon}\left(x,\nabla u_{\epsilon}\right),\nabla u_{\epsilon}\right)dx
\end{align*}
goes to 0, as $\epsilon\rightarrow0$.  This is done in four steps.
	
In what follows, we use the following notation $$\displaystyle \xi_{\epsilon}^{i}=\frac{1}{\left|Y_{\epsilon}^{i}\right|}\int_{Y_{\epsilon}^{i}}\nabla u dx.$$

\textbf{STEP~1}
\vspace{1mm}

Let us prove that
\begin{equation}
	\label{First}
	\displaystyle
	\int_{\Omega}\left(A_{\epsilon}\left(x,p_{\epsilon}\left(x,M_{\epsilon}\nabla u\right)\right),p_{\epsilon}\left(x,M_{\epsilon}\nabla u\right)\right)dx \rightarrow \int_{\Omega}\left(b(\nabla u),\nabla u\right)dx
\end{equation}
as $\epsilon\rightarrow0$.

\begin{proof}
From (\ref{inner product a with p}) and (\ref{approximation}), we obtain
\begin{align*}
	\displaystyle
	& \int_{\Omega}\left(A_{\epsilon}\left(x,p_{\epsilon}\left(x,M_{\epsilon}\nabla u(x)\right)\right),p_{\epsilon}\left(x,M_{\epsilon}\nabla u(x)\right)\right)dx\\
	&\quad=\int_{\Omega_{\epsilon}}\left(A_{\epsilon}\left(x,p_{\epsilon}\left(x,M_{\epsilon}\nabla u(x)\right)\right),p_{\epsilon}\left(x,M_{\epsilon}\nabla u(x)\right)\right)dx\\
	&\quad= \sum_{i\in I_{\epsilon}}\int_{Y_{\epsilon}^{i}}\left(A\left(\frac{x}{\epsilon},p\left(\frac{x}{\epsilon},\xi_{\epsilon}^{i}\right)\right),p\left(\frac{x}{\epsilon},\xi_{\epsilon}^{i}\right)\right)dx\\
	&\quad=\epsilon^{n}\sum_{i\in I_{\epsilon}}\int_{Y}\left(A\left(y,p\left(y,\xi_{\epsilon}^{i}\right)\right),p\left(y,\xi_{\epsilon}^{i}\right)\right)dy\\
	&\quad= \sum_{i\in I_{\epsilon}}\int_{\Omega}\chi_{Y_{\epsilon}^{i}}(x)\left(b(\xi_{\epsilon}^{i}),\xi_{\epsilon}^{i}\right)dx=\int_{\Omega}\left(b(M_{\epsilon}\nabla u(x)),M_{\epsilon}\nabla u(x)\right)dx.		
\end{align*}

By (\ref{Conb}), the definition of $q_{1}$, and H\"older's inequality we have 
\begin{align*}
	\displaystyle
	&\int_{\Omega}\left|b(M_{\epsilon}\nabla u(x))-b(\nabla u(x))\right|^{q_{1}}dx\\
	&\quad\leq C\left[\left(\int_{\Omega}\left|M_{\epsilon}\nabla u(s)-\nabla u(s)\right|^{p_{2}}dx\right)^{\frac{1}{(p_{2}-1)^{2}}}\right.\\
	&\qquad+\left.\left(\int_{\Omega}\left|M_{\epsilon}\nabla u(x)-\nabla u(x)\right|^{p_{2}}dx\right)^{\frac{1}{(p_{2}-1)(p_{1}-1)}}\right]	
\end{align*}	

From Property~1 of $M_{\epsilon}$, we obtain that
\begin{equation}
	\label{step1-1}
	\displaystyle b(M_{\epsilon}\nabla u)\rightarrow b(\nabla u)  \text{  in $L^{q_{1}}(\Omega;\mathbb{R}^{n})$}\text{, as $\epsilon\rightarrow0$}.
\end{equation}	

Now, (\ref{First}) follows from (\ref{step1-1}) since $M_{\epsilon}\nabla u\rightarrow \nabla u$ in $L^{p_2}(\Omega;\mathbb{R}^n)$, so
\begin{align*}
	\displaystyle
	\int_{\Omega}\left(A_{\epsilon}\left(x,p_{\epsilon}\left(x,M_{\epsilon}\nabla u(x)\right)\right),p_{\epsilon}\left(x,M_{\epsilon}\nabla u(x)\right)\right)dx
	& = \int_{\Omega}\left(b(M_{\epsilon}\nabla u(x),M_{\epsilon}\nabla u(x)\right)dx\\
	& \rightarrow \int_{\Omega}\left(b(\nabla u(x)),\nabla u(x)\right)dx,
\end{align*}	
as $\epsilon\rightarrow0$.
\end{proof}

\textbf{STEP~2}
\vspace{1mm}

We now show that
\begin{equation}
	\label{Second}
	\displaystyle \int_{\Omega}\left(A_{\epsilon}\left(x,p_{\epsilon}\left(x,M_{\epsilon}\nabla u(x)\right)\right),\nabla u_{\epsilon}(x)\right)dx
	\rightarrow \int_{\Omega}\left(b(\nabla u(x)),\nabla u(x)\right)dx 
\end{equation}
as $\epsilon\rightarrow0$.

\begin{proof}
Let $\delta>0$.  From Theorem~\ref{regularity} we have $\nabla u\in L^{p_{2}}(\Omega;\mathbb{R}^{n})$ and there exists 
a simple function $\Psi$ satisfying the assumptions of Lemma~\ref{lemma3} such that 
\begin{equation}
	\label{approximation with simple function of Du}
	\displaystyle
	\left\|\nabla u-\Psi\right\|_{L^{p_{2}}(\Omega;\mathbb{R}^{n})}\leq\delta.
\end{equation}

Let us write
\begin{align*}
	\displaystyle &\int_{\Omega}\left(A_{\epsilon}\left(x,p_{\epsilon}\left(x,M_{\epsilon}\nabla u(x)\right)\right),\nabla u_{\epsilon}(x)\right)dx \\&= \int_{\Omega}\left(A_{\epsilon}\left(x,p_{\epsilon}\left(x,\Psi\right)\right),\nabla u_{\epsilon}\right)dx 
	+ \int_{\Omega}\left(A_{\epsilon}\left(x,p_{\epsilon}\left(x,M_{\epsilon}\nabla u\right)\right)-A_{\epsilon}\left(x,p_{\epsilon}\left(x,\Psi \right)\right),\nabla u_{\epsilon}\right)dx.
\end{align*}

We first show that $$\int_{\Omega}\left(A_{\epsilon}\left(x,p_{\epsilon}\left(x,\Psi(x)\right)\right),\nabla u_{\epsilon}(x)\right)dx\rightarrow\int_{\Omega}\left(b(\Psi(x)),\nabla u(x)\right)dx \text{ as $\epsilon\rightarrow0$}.$$  
We have $$\int_{\Omega}\left(A_{\epsilon}\left(x,p_{\epsilon}\left(x,\Psi(x)\right)\right),\nabla u_{\epsilon}(x)\right)dx=\sum_{j=0}^{m}\int_{\Omega_{j}}\left(A_{\epsilon}\left(x,p_{\epsilon}\left(x,\eta_{j}\right)\right),\nabla u_{\epsilon}(x)\right)dx.$$
	
Now from (\ref{p5}), we have that $\displaystyle A_{\epsilon}\left(\cdot,p_{\epsilon}\left(\cdot,\eta_{j}\right)\right)\rightharpoonup b(\eta_{j})\in L^{q_{2}}(\Omega_{j};\mathbb{R}^{n}),$ and by (\ref{div with p}), $\displaystyle \int_{\Omega_{j}}\left(A_{\epsilon}\left(x,p_{\epsilon}\left(x,\eta_{j}\right)\right),\nabla\varphi(x)\right)dx=0,$ for $\varphi\in W_{0}^{1,p_{1}}(\Omega_{j})$.

Take $\varphi=\delta u_{\epsilon}$, with $\delta\in C_{0}^{\infty}(\Omega_{j})$ to get $$0=\int_{\Omega_{j}}\left(A_{\epsilon}\left(x,p_{\epsilon}\left(x,\eta_{j}\right)\right),(\nabla\delta)u_{\epsilon}\right)dx + \int_{\Omega_{j}}\left(A_{\epsilon}\left(x,p_{\epsilon}\left(x,\eta_{j}\right)\right),(\nabla u_{\epsilon})\delta\right)dx.$$
	
Taking the limit as $\epsilon\rightarrow0$, and using the fact that $u^{\epsilon}\rightharpoonup u$ in $W_{0}^{1,p_{1}}(\Omega)$ 
and (\ref{p5}), we have by Lemma~\ref{dunfordpettis} that $$\int_{\Omega_{j}}g_{j}(x)\delta(x)dx =\lim_{\epsilon\rightarrow0}\int_{\Omega_{j}}\left(A_{\epsilon}\left(x,p_{\epsilon}\left(x,\eta_{j}\right)\right),(\nabla u_{\epsilon})\delta\right)dx=\int_{\Omega_{j}}\left(b(\eta_{j}),(\nabla u)\delta\right)dx$$

Therefore, we may conclude that $g_{j}=\left(b(\eta_{j}),\nabla u\right)$, so $$\sum_{j=0}^{n}\int_{\Omega_{j}}\left(A_{\epsilon}\left(x,p_{\epsilon}\left(x,\eta_{j}\right)\right),\nabla u_{\epsilon}(x)\right)dx\rightarrow\sum_{j=0}^{n}\int_{\Omega_{j}}\left(b(\eta_{j}),\nabla u(x)\right)dx\text{, as $\epsilon\rightarrow0$.}$$ 

Thus, we get $$\int_{\Omega}\left(A_{\epsilon}\left(x,p_{\epsilon}\left(x,\Psi(x)\right)\right),\nabla u_{\epsilon}(x)\right)dx\rightarrow\int_{\Omega}\left(b(\Psi(x)),\nabla u(x)\right)dx\text{, as $\epsilon\rightarrow0$.}$$ 

On the other hand, let us estimate $$\int_{\Omega}\left(A_{\epsilon}\left(x,p_{\epsilon}\left(x,M_{\epsilon}\nabla u(x)\right)\right)-A_{\epsilon}\left(x,p_{\epsilon}\left(x,\Psi(x)\right)\right),\nabla u_{\epsilon}(x)\right)dx.$$  
	
By (\ref{ConA}) and H\"older's inequality we obtain
\begin{align}
	\label{step2-1}
	\displaystyle 
	&\left|\int_{\Omega}\left(A_{\epsilon}\left(x,p_{\epsilon}\left(x,M_{\epsilon}\nabla u(x)\right)\right)-A_{\epsilon}\left(x,p_{\epsilon}\left(x,\Psi(x) \right)\right),\nabla u_{\epsilon}(x)\right)dx\right|\\
	&\quad\leq C\left(\int_{\Omega}\chi_{1}^{\epsilon}(x)\left|p_{\epsilon}\left(x,M_{\epsilon}\nabla u\right)-p_{\epsilon}\left(x,\Psi \right)\right|^{p_{1}}dx\right)^{\frac{1}{p_{1}}}\left(\int_{\Omega}\chi_{1}^{\epsilon}(x)\left|\nabla u_{\epsilon}\right|^{p_{1}}dx\right)^{\frac{1}{p_{1}}}\notag\\
	&\qquad\quad\times\left(\int_{\Omega}\chi_{1}^{\epsilon}(x)\left(1+\left|p_{\epsilon}\left(x,M_{\epsilon}\nabla u\right)\right|^{p_{1}}+\left|p_{\epsilon}\left(x,\Psi \right)\right|^{p_{1}}\right)dx\right)^{\frac{p_{1}-2}{p_{1}}}\notag\\
	&\qquad+C\left(\int_{\Omega}\chi_{2}^{\epsilon}(x)\left|p_{\epsilon}\left(x,M_{\epsilon}\nabla u\right)-p_{\epsilon}\left(x,\Psi \right)\right|^{p_{2}}dx\right)^{\frac{1}{p_{2}}}\left(\int_{\Omega}\chi_{2}^{\epsilon}(x)\left|\nabla u_{\epsilon}\right|^{p_{2}}dx\right)^{\frac{1}{p_{2}}}\notag\\
	&\qquad\quad\times \left(\int_{\Omega}\chi_{2}^{\epsilon}(x)\left(1+\left|p_{\epsilon}\left(x,M_{\epsilon}\nabla u\right)\right|^{p_{2}}+\left|p_{\epsilon}\left(x,\Psi \right)\right|^{p_{2}}\right)dx\right)^{\frac{p_{2}-2}{p_{2}}}\notag
\end{align}

Applying (\ref{aprioribound}), (\ref{uniform boundedness of p at M}), and Lemma~\ref{lemma1} to the right 
hand side of (\ref{step2-1}), we obtain
\begin{align}
	\label{step2-2}
	\displaystyle
	& \left|\int_{\Omega}\left(A_{\epsilon}\left(x,p_{\epsilon}\left(x,M_{\epsilon}\nabla u(x)\right)\right)-A_{\epsilon}\left(x,p_{\epsilon}\left(x,\Psi(x) \right)\right),\nabla u_{\epsilon}(x)\right)dx\right|\\
	&\quad\leq C\left[\left(\int_{\Omega}\chi_{1}^{\epsilon}(x)\left|p_{\epsilon}\left(x,M_{\epsilon}\nabla u(x)\right)-p_{\epsilon}\left(x,\Psi(x)\right)\right|^{p_{1}}dx\right)^{\frac{1}{p_{1}}}\right.\notag\\
	&\qquad\left.+\left(\int_{\Omega}\chi_{2}^{\epsilon}(x)\left|p_{\epsilon}\left(x,M_{\epsilon}\nabla u(x)\right)-p_{\epsilon}\left(x,\Psi(x) \right)\right|^{p_{2}}dx\right)^{\frac{1}{p_{2}}}\right]\notag
\end{align}
	
Applying Lemma~\ref{lemma3} and (\ref{approximation with simple function of Du}) to (\ref{step2-2}), we discover that	
\begin{align}
	\label{step2-3}
	\displaystyle & \limsup_{\epsilon\rightarrow0}\left|\int_{\Omega}\left(A_{\epsilon}\left(x,p_{\epsilon}\left(x,M_{\epsilon}\nabla u(x)\right)\right)-A_{\epsilon}\left(x,p_{\epsilon}\left(x,\Psi(x)\right)\right),\nabla u_{\epsilon}(x)\right)dx\right|\\
	&\quad\leq C\left[\left(\delta^{q_{1}}+\delta^{q_{2}}\right)^{\frac{1}{p_{1}}} + \left(\delta^{q_{1}}+\delta^{q_{2}}\right)^{\frac{1}{p_{2}}}\right],\notag
\end{align}	
where $C$ is independent of $\delta$.  Since $\delta$ is arbitrary we conclude that the limit on the left hand side of (\ref{step2-3}) is equal to $0$.

Finally, using the continuity of $b$ and H\"older's inequality we obtain $$\left|\int_{\Omega}\left(b(\nabla u(x))-b(\Psi(x)),\nabla u(x)\right)dx\right|\leq C\left[\delta^{\frac{1}{(p_{1}-1)(p_{2}-1)}}+\delta^{\frac{1}{(p_{2}-1)^{2}}}\right]^{\frac{1}{q_{1}}},$$
where $C$ does not depend on $\delta$. 

Step~2 is proved noticing that $\delta$ can be taken arbitrarily small.
\end{proof}

\textbf{STEP~3}
\vspace{1mm}
	
We will show that
\begin{equation}
\label{Third}
	\displaystyle\int_{\Omega}\left(A_{\epsilon}\left(x,\nabla u_{\epsilon}(x)\right),p_{\epsilon}\left(x,M_{\epsilon}\nabla u(x)\right)\right)dx
	 \rightarrow\int_{\Omega}\left(b(\nabla u(x)),\nabla u(x)\right)dx
\end{equation}
as $\epsilon\rightarrow0$.	

\begin{proof}
Let $\delta>0$.  As in the proof of Step~2, assume $\Psi$ is a simple function satisfying assumptions of 
Lemma~\ref{lemma3} and such that $\displaystyle \left\|\nabla u-\Psi\right\|_{L^{p_{2}}(\Omega;\mathbb{R}^{n})}<\delta$.
		
Let us write 
\begin{align*}
	\displaystyle
	&\int_{\Omega}\left(A_{\epsilon}\left(x,\nabla u_{\epsilon}(x)\right),p_{\epsilon}\left(x,M_{\epsilon}\nabla u(x)\right)\right)dx\\
	&\quad=\int_{\Omega}\left(A_{\epsilon}\left(x,\nabla u_{\epsilon}(x)\right),p_{\epsilon}\left(x,\Psi(x)\right)\right)dx\\
	&\qquad+\int_{\Omega}\left(A_{\epsilon}\left(x,\nabla u_{\epsilon}(x)\right),p_{\epsilon}\left(x,M_{\epsilon}\nabla u(x)\right)-p_{\epsilon}\left(x,\Psi(x)\right)\right)dx.	
\end{align*}
	
We first show that $$\int_{\Omega}\left(A_{\epsilon}\left(x,\nabla u_{\epsilon}(x)\right),p_{\epsilon}\left(x,\Psi(x)\right)\right)dx\rightarrow\int_{\Omega}\left(b\left(\nabla u(x)\right),\Psi(x)\right)dx.$$

We start by writing $$\int_{\Omega}\left(A_{\epsilon}\left(x,\nabla u_{\epsilon}(x)\right),p_{\epsilon}\left(x,\Psi(x)\right)\right)dx= \sum_{j=0}^{m}\int_{\Omega_{j}}\left(A_{\epsilon}\left(x,\nabla u_{\epsilon}(x)\right),p_{\epsilon}\left(x,\eta_{j}\right)\right)dx.$$

From Lemma~\ref{dunfordpettis}, up to a subsequence, $\left(A_{\epsilon}\left(\cdot,\nabla u_{\epsilon}\right),p_{\epsilon}\left(\cdot,\eta_{j}\right)\right)$ converges weakly to a function $h_{j}\in L^{1}(\Omega_{j};\mathbb{R})$, as $\epsilon\rightarrow0$.

By Theorem~\ref{homogenization}, we have $\displaystyle A_{\epsilon}\left(\cdot,\nabla u_{\epsilon}\right)\rightharpoonup b(\nabla u)\in L^{q_{2}}(\Omega;\mathbb{R}^{n})$ and $$\displaystyle -div\left(A_{\epsilon}\left(x,\nabla u_{\epsilon}\right)\right)=f=-div\left(b(\nabla u)\right).$$ 
	
From (\ref{p3}), $p_{\epsilon}$ satisfies $\displaystyle p_{\epsilon}(\cdot,\eta_{j})\rightharpoonup\eta_{j}$ in $L^{p_{1}}(\Omega_{j},\mathbb{R}^{n})$. 
	
Arguing as in Step~2, we find that $\displaystyle \left(A_{\epsilon}\left(x,\nabla u_{\epsilon}(x)\right),p_{\epsilon}\left(x,\eta_{j}\right)\right) \rightharpoonup \left(b(\nabla u(x)),\eta_{j}\right)$ in $D^{'}(\Omega_{j})$, as $\epsilon\rightarrow0$.

Therefore, we may conclude that $h_{j}=\left(b(\nabla u),\eta_{j}\right)$, and hence, $$\sum_{j=0}^{n}\int_{\Omega_{j}}\left(A_{\epsilon}\left(x,\nabla u_{\epsilon}(x)\right),p_{\epsilon}\left(x,\eta_{j}\right)\right)dx\rightarrow\sum_{j=0}^{n}\int_{\Omega_{j}}\left(b(\nabla u(x)),\eta_{j}\right)dx\text{, as $\epsilon\rightarrow0$.}$$ 

Thus, we get $$\displaystyle \int_{\Omega}\left(A_{\epsilon}\left(x,\nabla u_{\epsilon}(x)\right),p_{\epsilon}\left(x,\Psi(x)\right)\right)dx\rightarrow\int_{\Omega}\left(b(\nabla u(x)),\Psi(x)\right)dx\text{, as $\epsilon\rightarrow0.$}$$

Moreover, applying H\"older's inequality and (\ref{ConA}) we have 
\begin{align*}
	\displaystyle
	& \left|\int_{\Omega}\left(A_{\epsilon}\left(x,\nabla u_{\epsilon}(x)\right),p_{\epsilon}\left(x,M_{\epsilon}\nabla u(x)\right)-p_{\epsilon}\left(x,\Psi(x)\right)\right)dx\right|\\
	&\quad\leq C\left[\left(\int_{\Omega}\chi_{1}^{\epsilon}\left(1+\left|\nabla u_{\epsilon}\right|\right)^{p_{1}}\right)^{\frac{1}{q_{2}}}\left(\int_{\Omega}\chi_{1}^{\epsilon}\left|p_{\epsilon}\left(x,M_{\epsilon}\nabla u\right)-p_{\epsilon}\left(x,\Psi\right)\right|^{p_{1}}dx\right)^{\frac{1}{p_{1}}}\right.\\
	&\qquad+\left.\left(\int_{\Omega}\chi_{2}^{\epsilon}\left(1+\left|\nabla u_{\epsilon}\right|\right)^{p_{2}}\right)^{\frac{1}{q_{1}}}\left(\int_{\Omega}\chi_{2}^{\epsilon}\left|p_{\epsilon}\left(x,M_{\epsilon}\nabla u\right)-p_{\epsilon}\left(x,\Psi\right)\right|^{p_{2}}dx\right)^{\frac{1}{p_{2}}}\right]
\end{align*}
	
As in the proof of Step~2 we see that $$\limsup_{\epsilon\rightarrow0}\left|\int_{\Omega}\left(A_{\epsilon}\left(x,\nabla u_{\epsilon}\right),p_{\epsilon}\left(x,M_{\epsilon}\nabla u\right)-p_{\epsilon}\left(x,\Psi\right)\right)dx\right|\leq C\left(\delta^{\frac{1}{p_{1}-1}}+\delta^{\frac{1}{p_{2}-1}}\right),$$ where $C$ does not depend on $\delta$.

Hence, proceeding as in Step~2, we find that
\begin{align*}
	\displaystyle				
	& \limsup_{\epsilon\rightarrow0}\left|\int_{\Omega}\left(A_{\epsilon}\left(x,\nabla u_{\epsilon}\right),p_{\epsilon}\left(x,M_{\epsilon} \nabla u\right)\right)dx-\int_{\Omega}\left(b(\nabla u),\nabla u\right)dx\right|\\
	& \quad\leq C\left(\delta^{\frac{1}{p_{1}-1}}+\delta^{\frac{1}{p_{2}-1}}+0+\left\|b(\nabla u)\right\|_{L^{q_{2}}(\Omega,\mathbb{R}^{n})}\delta^{\frac{1}{p_{1}}}\right),
\end{align*}
where $C$ is independent of $\delta$.  Now since $\delta$ is arbitrarily small, the proof of Step~3 is complete.
\end{proof}

\textbf{STEP~4}
\vspace{1mm}
	
Finally, let us prove that
\begin{equation}
	\label{Fourth}
	\displaystyle \int_{\Omega}\left(A_{\epsilon}\left(x,\nabla u_{\epsilon}(x)\right),\nabla u_{\epsilon}(x)\right)dx
	 \rightarrow\int_{\Omega}\left(b(\nabla u(x)),\nabla u(x)\right)dx \text{, as $\epsilon\rightarrow0$}.
\end{equation}
	
\begin{proof}
	Since
	\begin{equation}
		\displaystyle \int_{\Omega}\left(A_{\epsilon}\left(x,\nabla u_{\epsilon}\right),\nabla u_{\epsilon}\right)dx=\left\langle -div\left(A_{\epsilon}\left(x,\nabla u_{\epsilon}\right)\right),u_{\epsilon}\right\rangle=\left\langle f,u_{\epsilon}\right\rangle,
	\end{equation}	
	\begin{equation}
		\displaystyle \int_{\Omega}\left(b(\nabla u),\nabla u\right)dx=\left\langle -div\left(b\left(\nabla u\right)\right),u\right\rangle=\left\langle f,u\right\rangle,
	\end{equation}
and $u_{\epsilon}\rightharpoonup u$ in $W^{1,p_{1}}(\Omega)$, the result follows immediately.
\end{proof}	

Finally, Theorem~\ref{corrector} follows from (\ref{First}), (\ref{Second}), (\ref{Third}) and (\ref{Fourth}).
\end{proof}

\subsection{Proof of the Lower Bound on the Amplification of the Macroscopic Field by the Microstructure}
\label{proof fluctuations}

The sequence $\left\{\chi_{i}^{\epsilon}(x)\nabla u_{\epsilon}(x)\right\}_{\epsilon>0}$ has a Young measure 
$\nu^{i}=\left\{\nu_{x}^{i}\right\}_{x\in\Omega}$ associated to it (see Theorem~6.2 and the discussion following 
in \cite{Pedregal1997}), for $i=1,2$.

As a consequence of Theorem~\ref{corrector} proved in the previous section, we have that  $$\left\|\chi_{i}^{\epsilon}(x)p\left(\frac{x}{\epsilon},M_{\epsilon}(\nabla u)(x)\right)- \chi_{i}^{\epsilon}(x)\nabla u_{\epsilon}(x)\right\|_{\textbf{L}^{p_{i}}(\Omega;\mathbb{R}^{n})}\rightarrow0,$$
as $\epsilon\rightarrow0$, which implies that the sequences$$\left\{\chi_{i}^{\epsilon}(x)p\left(\frac{x}{\epsilon},M_{\epsilon}(\nabla u)(x)\right)\right\}_{\epsilon>0} \text{  and  } \left\{\chi_{i}^{\epsilon}(x)\nabla u_{\epsilon}(x)\right\}_{\epsilon>0}$$share the same Young measure (see Lemma~6.3 of \cite{Pedregal1997}), for $i=1,2$.
  
The next lemma identifies the Young measure $\nu^i$.
  
\begin{lemma}
\label{lemmafluctuations}
For all $\phi\in C_{0}(\mathbb{R}^{n})$ and for all $\zeta\in C^{\infty}_{0}(\mathbb{R}^{n})$, we have
\begin{equation}
	\label{field-1}
 	\displaystyle
	\int_{\Omega}\zeta(x)\int_{\textbf{R}^{n}}\phi(\lambda)d\nu_{x}^{i}(\lambda)dx=\int_{\Omega}\zeta(x)\int_{Y}\phi(\chi_{i}(y)p(y,\nabla u(x)))dydx
\end{equation}
\end{lemma}

\begin{proof}
	To prove (\ref{field-1}), we will show that given $\phi\in C_0(\mathbb{R}^n)$  and $\zeta\in C_0^\infty(\mathbb{R}^n)$ that
\begin{align}
	\label{field-2}
	 &\lim_{\epsilon\rightarrow0}\int_{\Omega}\zeta(x)\phi\left(\chi_{i}^{\epsilon}(x)p\left(\frac{x}{\epsilon},M_{\epsilon}\left(\nabla u\right)(x)\right)\right)dx\notag\\
	 &\quad=\int_{\Omega}\zeta(x)\int_{Y}\phi(\chi_{i}(y)p(y,\nabla u(x)))dydx.
\end{align}

We consider the difference
\begin{align}
	\label{difference}
	\displaystyle 
	& \left|\int_{\Omega}\zeta(x)\phi\left(\chi_{i}\left(\frac{x}{\epsilon}\right)p\left(\frac{x}{\epsilon},M_{\epsilon}(\nabla u)(x)\right)\right)dx - \int_{\Omega}\zeta(x)\int_{Y}\phi\left(\chi_{i}\left(y\right)p\left(y,\nabla u(x)\right)\right)dydx\right|\notag\\
	&\quad\leq\left|\sum_{i\in I_{\epsilon}}\int_{Y_{\epsilon}^{i}}\zeta(x)\phi\left(\chi_{i}\left(\frac{x}{\epsilon}\right)p\left(\frac{x}{\epsilon},\xi_{\epsilon}^{i}\right)\right)dx - \int_{\Omega_{\epsilon}}\zeta(x)\int_{Y}\phi\left(\chi_{i}\left(y\right)p\left(y,\nabla u(x)\right)\right)dydx\right|\notag\\
	&\qquad+C\left|\Omega\setminus\Omega_{\epsilon}\right|.
\end{align}

Note that the term $C\left|\Omega\setminus\Omega_{\epsilon}\right|$ goes to $0$, as $\epsilon\rightarrow0$.  Now set $x_{\epsilon}^{i}$ to be the center of $Y_{\epsilon}^{i}$. On the first integral use the change of variables $x=x_{\epsilon}^{i}+\epsilon y$, where $y$ belongs to $Y$, and since $dx=\epsilon^{n}dy$, we get
\begin{align*}
	\displaystyle 
	& \left|\sum_{i\in I_{\epsilon}}\int_{Y_{\epsilon}^{i}}\zeta(x)\phi\left(\chi_{i}\left(\frac{x}{\epsilon}\right)p\left(\frac{x}{\epsilon},\xi_{\epsilon}^{i}\right)\right)dx - \sum_{i\in I_{\epsilon}}\int_{Y_{\epsilon}^{i}}\zeta(x)\int_{Y}\phi\left(\chi_{i}\left(y\right)p\left(y,\nabla u(x)\right)\right)dydx\right|\\
	&\quad=\left|\sum_{i\in I_{\epsilon}}\epsilon^{n}\int_{Y}\zeta(x_{\epsilon}^{i}+\epsilon y)\phi\left(\chi_{i}\left(y\right)p\left(y,\xi_{\epsilon}^{i}\right)\right)dy\right.\\
	&\qquad\left. - \sum_{i\in I_{\epsilon}}\int_{Y_{\epsilon}^{i}}\zeta(x)\int_{Y}\phi\left(\chi_{i}\left(y\right)p\left(y,\nabla u(x)\right)\right)dydx\right|
	&\intertext{Applying Taylor's expansion for $\zeta$, we have}
	&\quad\leq \left|\sum_{i\in I_{\epsilon}}\int_{Y_{\epsilon}^{i}}\int_{Y}\left(\zeta(x)+CO(\epsilon)\right)\left[\phi\left(\chi_{i}\left(y\right)p\left(y,\xi_{\epsilon}^{i}\right)\right)- \phi\left(\chi_{i}\left(y\right)p\left(y,\nabla u(x)\right)\right)\right]dydx\right|\\
	&\qquad+CO(\epsilon)\\
	&\quad\leq\left|\int_{\Omega_{\epsilon}}\left|\zeta(x)\right|\int_{Y}\left|\phi\left(\chi_{i}\left(y\right)p\left(y,M_{\epsilon}\nabla u(x)\right)\right)- \phi\left(\chi_{i}\left(y\right)p\left(y,\nabla u(x)\right)\right)\right|dydx\right|\\
	&\qquad+CO(\epsilon)
	&\intertext{Because of the uniform Lipschitz continuity of $\phi$, we get}
	&\quad\leq C\left|\int_{\Omega_{\epsilon}}\left|\zeta(x)\right|\int_{Y}\left|p\left(y,M_{\epsilon}\nabla u(x)\right)- p\left(y,\nabla u(x)\right)\right|dydx\right| + CO(\epsilon)
	&\intertext{By H\"older's inequality twice and Lemma~\ref{lemma2}, we have}
	&\quad\leq C\left\{\left(\int_{\Omega_{\epsilon}}\left|\zeta(x)\right|^{q_{2}}dx\right)^{1/q_{2}}\left[\int_{\Omega_{\epsilon}}\Big(\left|M_{\epsilon}\nabla u(x)-\nabla u(x)\right|^{\frac{p_{1}}{p_{1}-1}}\theta_{1}^{\frac{1}{p_{1}-1}}\right.\right.\\
	&\qquad\times \left(1+\left|M_{\epsilon}\nabla u(x)\right|^{p_{1}}\theta_{1}+\left|M_{\epsilon} \nabla u(x)\right|^{p_{2}}\theta_{2}+\left|\nabla u(x)\right|^{p_{1}}\theta_{1}+\left|\nabla u(x)\right|^{p_{2}}\theta_{2}\right)^{\frac{p_{1}-2}{p_{1}-1}}\\
	&\quad+\left|M_{\epsilon}\nabla u(x)-\nabla u(x)\right|^{\frac{p_{2}}{p_{2}-1}}\theta_{2}^{\frac{1}{p_{2}-1}}\\
	&\qquad\times\left.\left(1+\left|M_{\epsilon}\nabla u(x)\right|^{p_{1}}\theta_{1}+\left|M_{\epsilon}\nabla u(x)\right|^{p_{2}}\theta_{2}+\left|\nabla u(x)\right|^{p_{1}}\theta_{1}+\left|\nabla u(x)\right|^{p_{2}}\theta_{2}\right)^{\frac{p_{2}-2}{p_{2}-1}}\Big)dx\right]^{1/p_{1}}\\
	&\quad+\left(\int_{\Omega_{\epsilon}}\left|\zeta(x)\right|^{q_{1}}dx\right)^{1/q_{1}}\left[\int_{\Omega_{\epsilon}}\left(\left|M_{\epsilon}\nabla u(x)-\nabla u(x)\right|^{\frac{p_{1}}{p_{1}-1}}\theta_{1}^{\frac{1}{p_{1}-1}}\right.\right.\\
	&\qquad\times\left(1+\left|M_{\epsilon}\nabla u(x)\right|^{p_{1}}\theta_{1}+\left|M_{\epsilon} \nabla u(x)\right|^{p_{2}}\theta_{2}+\left|\nabla u(x)\right|^{p_{1}}\theta_{1}+\left|\nabla u(x)\right|^{p_{2}}\theta_{2}\right)^{\frac{p_{1}-2}{p_{1}-1}}\\
	&\quad+\left|M_{\epsilon}\nabla u(x)-\nabla u(x)\right|^{\frac{p_{2}}{p_{2}-1}}\theta_{2}^{\frac{1}{p_{2}-1}}\\
	&\qquad\times\left.\left.\left.\left(1+\left|M_{\epsilon}\nabla u(x)\right|^{p_{1}}\theta_{1}+\left|M_{\epsilon}\nabla u(x)\right|^{p_{2}}\theta_{2}+\left|\nabla u(x)\right|^{p_{1}}\theta_{1}+\left|\nabla u(x)\right|^{p_{2}}\theta_{2}\right)^{\frac{p_{2}-2}{p_{2}-1}}\right)dx\right]^{1/p_{2}}\right\}\\
	&\quad+CO(\epsilon)
	&\intertext{Applying H\"older's inequality again, we get} 
	&\quad\leq C\left[\left(\int_{\Omega_{\epsilon}}\left|M_{\epsilon}\nabla u(x)-\nabla u(x)\right|^{p_{1}}dx\right)^{\frac{1}{p_{1}-1}}\right.\\
	&\qquad+\left. \left(\int_{\Omega_{\epsilon}}\left|M_{\epsilon}\nabla u(x)-\nabla u(x)\right|^{p_{2}}dx\right)^{\frac{1}{p_{2}-1}}\right]^{1/p_{1}}\\
	&\quad + C\left[\left(\int_{\Omega_{\epsilon}}\left|M_{\epsilon}\nabla u(x)-\nabla u(x)\right|^{p_{1}}dx\right)^{\frac{1}{p_{1}-1}}dx\right.\\
	&\qquad \left. +\left(\int_{\Omega_{\epsilon}}\left|M_{\epsilon}\nabla u(x)-\nabla u(x)\right|^{p_{2}}dx\right)^{\frac{1}{p_{2}-1}}\right]^{1/p_{2}}+CO(\epsilon).
\end{align*}

Finally, from the approximation property of $M_{\epsilon}$ in Section~\ref{CorrectorSection}, as $\epsilon\rightarrow0$, we obtain (\ref{field-2}).

Therefore, from Proposition~4.4 of \cite{Pedregal1999} and (\ref{field-2}) we have
\begin{align*}
 	\displaystyle
	\int_{\Omega}\zeta(x)\int_{\textbf{R}^{n}}\phi(\lambda)d\nu_{x}^{i}(\lambda)dx
 	& = \int_{\Omega}\zeta(x)\int_{Y}\phi(\chi_{i}(y)p(y,\nabla u(x)))dydx\\
 	& =  \lim_{\epsilon\rightarrow0}\int_{\Omega}\zeta(x)\phi\left(\chi_{i}^{\epsilon}(x)p\left(\frac{x}{\epsilon},M_{\epsilon}(\nabla u)(x)\right)\right)dx\\
 	& \leq \lim_{\epsilon\rightarrow0}\int_{\Omega}\zeta(x)\phi\left(\chi_{i}^{\epsilon}(x)\nabla u_{\epsilon}(x)\right)dx, 
\end{align*}
for all $\phi\in C_{0}(\mathbb{R}^{n})$ and for all $\zeta\in C^{\infty}_{0}(\mathbb{R}^{n})$. 
\end{proof}

The proof of Theorem~\ref{fluctuations} follows from Lemma~\ref{lemmafluctuations} and Theorem~6.11 in \cite{Pedregal1997}.
	
\section{Summary}
\label{conclusions}

In this paper we consider a composite material made from two  materials with different power law behavior. The exponent of the power law is different for each material and taken to be $p_{1}$  in material one and $p_{2}$ in material two with $2\leq p_1< p_2<\infty$.
For this case we have introduced a   corrector theory for the strong approximation of fields inside these
composites, see Theorem~\ref{corrector}.   The correctors are then used to provide lower bounds on the local singularity strength inside micro-structured media. The bounds are multi-scale in nature and quantify the amplification of 
applied macroscopic fields by the microstructure, see Theorem~\ref{fluctuations}. 
These results are shown to hold for finely mixed periodic dispersions of inclusions and for layers.  Future work seeks to extend the analysis to multi-phase power law materials and for different regimes of exponents $p_1$ and $p_2$.
\bibliographystyle{plain}	
\bibliography{paper}
\end{document}